  \documentclass[reqno, 11 pt]{amsart}
  
\usepackage{amssymb}   
\usepackage{amsmath}   
\usepackage{euscript}   
\usepackage{hyperref} 
\usepackage{amsthm} 
\usepackage{cleveref} 
\usepackage{dsfont}  
\usepackage{amsfonts}    
\usepackage{marginnote}
\usepackage{latexsym} 
\usepackage{amsopn} 
\usepackage{geometry}
\usepackage{amscd} 
\usepackage{mathrsfs}
\usepackage{caption}
\usepackage{aurical}  
\usepackage[english]{babel}
\usepackage[utf8]{inputenc}
\usepackage{esint} 
\usepackage{graphicx}
\usepackage[dvipsnames]{xcolor}
\usepackage{multicol}
\usepackage{nomencl}
\makeglossary
\makenomenclature  
\usepackage{refcount}
\usepackage{float}
\usepackage{times}
\usepackage[titletoc]{appendix}

\newcommand{\comment}[1]{}

\newcommand{\pa}{\partial}

\definecolor{airforce}{rgb}{0.36, 0.74, 0.86}
 	\definecolor{blue-green}{rgb}{0.0, 0.57, 0.87}

\newcommand{\s}{{\sigma}}

\renewcommand{\Re}{\operatorname{Re}}
\renewcommand{\Im}{\operatorname{Im}}

\newcommand{\Id}{\mathds{1}}

\newcommand{\opw}{{Op^{\mathrm{W}}}}
\newcommand{\opbw}{{Op^{\mathrm{BW}}}}

\newcommand{\gr}[1]{\textbf{#1}}

\providecommand{\vect}[2]{{\bigl[\begin{smallmatrix}#1\\#2\end{smallmatrix}\bigr]}}   
\providecommand{\sm}[4]{{\bigl[\begin{smallmatrix}#1&#2\\#3&#4\end{smallmatrix}\bigr]}}

\newtheorem{theorem}{Theorem}[section]
\newtheorem*{thm*}{Theorem}
\newtheorem{proposition}[theorem]{Proposition}
\newtheorem{lemma}[theorem]{Lemma}

\newtheorem{cor}[theorem]{Corollary}
\newtheorem*{cor*}{Corollary}
\newtheorem{remark}[theorem]{Remark}
\newtheorem{ex}{Example}
\newtheorem{hyp}[theorem]{Hypothesis}
\newtheorem{definition}[theorem]{Definition}

\numberwithin{equation}{section}

\newcommand{\ii}{{\rm i}}

\newcommand{\x}{\xi}

\newcommand{\C}{{\mathbb C}}

\newcommand{\N}{{\mathbb N}}

\newcommand{\R}{{\mathbb R}}

\newcommand{\T}{{\mathbb T}}
\newcommand{\Z}{{\mathbb Z}}

\newcommand{\cH}{{\mathcal H}}

\usepackage{bm}

\newcommand{\0}{{(0)}}

\newcommand{\nnorm}[1]{{\left\vert\kern-0.25ex\left\vert\kern-0.25ex\left\vert #1 
    \right\vert\kern-0.25ex\right\vert\kern-0.25ex\right\vert}}

\usepackage{framed,enumitem}

\providecommand{\vect}[2]{{\bigl[\begin{smallmatrix}#1\\#2\end{smallmatrix}\bigr]}} 
  
\providecommand{\sm}[4]{{\bigl[\begin{smallmatrix}#1&#2\\#3&#4\end{smallmatrix}\bigr]}}

\makeatletter

\renewcommand{\tocsection}[3]{%
\indentlabel{\@ifnotempty{#2}{\bfseries\ignorespaces#1 #2\quad}}\bfseries#3}
\def\l@subsection{\@tocline{2}{0pt}{2.5pc}{5pc}{}}
\def\l@subsubsection{\@tocline{3}{0pt}{4.5pc}{5pc}{}}
\renewcommand\tocchapter[3]{%
  \indentlabel{\@ifnotempty{#2}{\ignorespaces#2.\quad}}#3%
}

\begin{document} 
 
\title
{Local well posedness for a system of quasilinear PDEs modelling suspension bridges}
\date{}

\author{Roberto Feola}
\address{\scriptsize{Dipartimento di Matematica e Fisica, Universit\`a degli Studi RomaTre, 
Largo San Leonardo Murialdo 1, 00144, Roma, Italy}}
\email{roberto.feola@uniroma3.it}

\author{Filippo  Giuliani}
\address{\scriptsize{Dipartimento di Matematica, Politecnico di Milano, 
Piazza Leonardo da Vinci 32, 20133, Milano, Italy}}
\email{filippo.giuliani@polimi.it}

\author{Felice  Iandoli}
\address{\scriptsize{Dipartimento di Matematica ed Informatica, 
Universit\`a della Calabria, Ponte Pietro Bucci,  87036, Rende, Italy}}
\email{felice.iandoli@unical.it}

\author{Jessica Elisa Massetti} 
\address{\scriptsize{Dipartimento di Matematica e Fisica, Universit\`a degli Studi RomaTre, 
Largo San Leonardo Murialdo 1, 00144, Roma, Italy}}
\email{jessicaelisa.massetti@uniroma3.it}

\keywords{Quasilinear beam-wave equations, local well posedness, energy method, paradifferential calculus, suspension bridges} 

\subjclass[2010]{35G55, 35A01, 35M11, 35S50   }

   
\begin{abstract}    
In this paper we provide a local well posedness result 
for a quasilinear beam-wave system of equations on a one-dimensional 
spatial domain under periodic and Dirichlet boundary conditions. 
This kind of systems provides a refined model for the time-evolution of suspension bridges, 
where the beam and wave equations describe respectively the longitudinal 
and torsional motion of the deck. The quasilinearity arises when one takes into account the nonlinear restoring action of deformable cables and hangers. 
To obtain the \emph{a priori} estimates for the solutions of the linearized equation we build a modified energy by means of paradifferential changes of variables. Then we construct the solutions of the nonlinear problem by using a quasilinear iterative scheme \emph{\`a la} Kato.
%
\end{abstract}

\maketitle

\setcounter{tocdepth}{1}
\tableofcontents


\section{Introduction}

In this paper we prove a result of local well posedness in Sobolev regularity for a wide class of coupled nonlinear beam-wave equations. We consider systems of the form

\begin{equation}\label{sistemaponte}
\begin{cases}
y_{tt}=\mathcal{B} y+F_1(x,y,y_x,y_{xx},\theta,\theta_x,\theta_{xx})+\alpha y_t+\gamma \mathtt{f}_{b}(t), \qquad 
\\
\theta_{tt}=\mathcal{W} \theta+F_2(x,y,y_x,y_{xx},\theta,\theta_x,\theta_{xx})+\beta \theta_t
+\delta \mathtt{f}_{w}(t)\,, 
\quad 
\end{cases}
\end{equation}
\begin{align*}
y=y(t, x), \quad \theta&=\theta(t, x)\,, \qquad t\in \R, \quad 
x\in [0, L],\\
&L>0, \qquad \alpha, \beta, \gamma, \delta\in \R,
\end{align*}
under Dirichlet and periodic 
boundary conditions,
where $\mathcal{B}$ and $\mathcal{W}$ 
are the following time-independent linear operators
with variable  coefficients 
\begin{equation}\label{opastratti}
\begin{aligned}
\mathcal{B}&=
-\,  b(x) \circ \pa_{x}^4 
+B\,,
\qquad
\mathcal{W}&=  c(x) \circ \pa^2_{x}\,  +C\,,
\end{aligned}
\end{equation}
for some smooth functions $b(x), c(x)$
and where $B$ and $C$ are linear operators mapping 
$H^{s+2}$ into $H^{s}$\footnote{Here $H^{s}$ denotes the standard Sobolev space. See \eqref{spazioHHss} for a detailed definition.} for all $s\geq 0$, while 
$\mathtt{f}_{b}, \mathtt{f}_{w}{\in C^{\infty}(\R; \R)}$ are smooth functions depending on time.

The nonlinearities 
 $F_i\in C^{\infty}(\mathbb{R}^7;\mathbb{R}),\,i=1, 2,$ are 
homogeneous polynomial of  degree $2$ in
all the variables $y,y_x,y_{xx}$ and $\theta,\theta_x,\theta_{xx}$. 
Therefore,  system \eqref{sistemaponte} is a coupled system of \textbf{quasilinear} PDEs in the following twofold sense: 
\begin{itemize}[leftmargin=*]
\item[$\bullet$] The highest order spatial derivatives operators $b(x)\partial_x^4$ and $-c(x)\partial_x^2$ have non-constant coefficients.
\item[$\bullet$]  The nonlinearity $F_2$ depends on the highest order derivative appearing in the linear part, namely $\theta_{xx}$.
 \end{itemize}

\smallskip

The equations \eqref{sistemaponte} are inspired by models 
that describe the time-evolution of \textit{suspension bridges}. A suspension bridge is a bridge which is constituted by the following main
elements: four towers, a rectangular deck, two sustaining cables and some hangers which link the deck to the main cables. The most common suspension bridges possess a central main span and two minor side ones.  In the present article we shall focus on the time-evolution of the main span ({we refer to \cite{Garrione2019} for a mathematical analysis on the complete structure}).

More specifically, we consider a \emph{fish-bone} model as proposed in \cite{BerchioGazzola2015}, namely we assume that the deck is a degenerate plate whose cross sections are free to rotate around the mid-line of the roadway.

We consider the line composed by the barycenters of the cross sections of the deck and we denote by $y=y(t, x)$ the vertical displacement of such a line. On the other hand, the variable $\theta=\theta(t, x)$ describes the angle of rotation of the cross sections with respect to the horizontal position. 
In this way, our model describes the nonlinear behavior of the bridge in a small torsional regime where, at the linear level, the longitudinal and torsional motions of the deck are governed by a beam equation and a wave equation respectively. 

Note that the terms $\alpha\, y_t$ and $\beta\, \theta_t$ represent some damping effect, while $\gamma\, \mathtt{f}_b(t)$ and $\delta\, \mathtt{f}_w(t)$ are forcing terms whose effect can be interpreted as the action of the wind.

{From a mathematical point of view, the main feature of system \eqref{sistemaponte} is the quasilinear nature of the equations. This arises from the fact that we consider \emph{deformable} (i.e. not rigid) main cables. This feature has been taken into account by Arioli-Gazzola in \cite{ArioliGazzola2017} for the derivation of a refined model for suspension bridges (we refer also to \cite{Crasta:2020wx} for a more recent model). We continue the discussion about the applications to suspension bridges in section \ref{sec:derivation}.
However, we point out that it is beyond of the scope of the present paper to provide 
 an exhaustive treatment of the wide literature on the mathematical analysis of suspension bridges, 
 therefore we refer the interested reader to the book \cite{Gazzola:book} and the work \cite{Falocchi} (and references therein).}

\smallskip
\noindent
 The main result of this paper concerns the existence and uniqueness of  classical solutions for system \eqref{sistemaponte} and it is stated in the next section.
  
\subsection{Main result: the local well posedness}
Let us introduce the main hypothesis under which  the local well posedness 
of system \eqref{sistemaponte} is guaranteed, under periodic boundary conditions, i.e.
$x\in \T=\R/2\pi \Z$. 
{For the sake of convenience}, we set $L=2\pi$ as the length of the deck in the case of periodic boundary conditions, and $L=\pi$ in the case of Dirichlet boundary conditions respectively. 
In this way both problems are reduced to the search of solutions for the system \eqref{sistemaponte} with spatial period $2\pi$. The result for Dirichlet boundary conditions will follow by requiring that solutions have to be odd in $x$. This is the content of main Theorem \ref{thm:main1} below and Corollary \ref{coro1}.
For this reason, from now on we consider the system of equations \eqref{sistemaponte} on the one-dimensional torus $\T:=\R/2\pi \Z$ and we will make widely use of Fourier and microlocal analysis.

By slight abuse of notation, in the following we shall denote by $H^{s}$ either $H^{s}(\T;\R)$ or $H^{s}(\T;\C)$, $s\in\R$, 
the standard Sobolev space of 
real or complex valued periodic functions, whenever the context hides no ambiguity. Finally, we denote by $B_R(H^s)$ the open ball of radius $R$ in the space $H^s$.
\smallskip

\begin{hyp}\label{hyp1}

\noindent \emph{(Ellipticity).}
There exist constants $\mathtt{c}_1,\mathtt{c}_2>0$
such that 
{
\begin{equation*}
\begin{aligned}
&b(x)>\mathtt{c}_1\,,\qquad c(x)>\mathtt{c}_2\,,\qquad \forall\, x\in \mathbb{T}\,.
\end{aligned}
\end{equation*}
}
\end{hyp}

Note that the assumption above is naturally verified in physical applications.

\begin{theorem}{\bf (Local well posedness).}\label{thm:main1}
Let $\alpha, \beta, \gamma, \delta\in \R$.
Assume Hypothesis \ref{hyp1} and let $s\geq \mathfrak{s}+2>5$. 
Then there exists $R>0$ depending on $\mathtt{c_2}$ and $F_2$ such that the following holds.

\noindent
For any initial conditions satisfying 
\begin{align*}
(y_0,y_1)\in (H^{s+1}\times H^{s-1})\cap B_{R}(H^6\times H^4),\\
(\theta_0,\theta_1)\in (H^{s+\frac{1}{2}}\times H^{s-\frac12})\cap B_{R}(H^{\frac{11}{2}}\times H^{\frac{9}{2}}),
\end{align*}
there exists  a time
\[
T=T(\alpha, \beta, \gamma, \delta; s,\|y_0\|_{H^{\mathfrak{s}+3}}, \|y_1\|_{H^{\mathfrak{s}+1}} , 
\|\theta_0\|_{H^{\mathfrak{s}+\frac{5}{2}}},
\|\theta_1\|_{H^{\mathfrak{s}+\frac{3}{2}}})>0
\]
and a unique solution $(y(t),\theta(t))$ 
of \eqref{sistemaponte} 
with 
\[
(y,\pa_{t}y)(0)=(y_0,y_1)\,, 
\quad
(\theta,\pa_{t}\theta)(0)=(\theta_0,\theta_1)
\]
such that 
\[
(y(t),\theta(t))\in C^0\big([0,T), H^{s+1}\times H^{s+\frac12}\big)
\cap C^1\big([0,T), H^{s-1}\times H^{s-\frac12}\big)\,.
\]
Moreover the solution map $(y_0,y_1,\theta_0,\theta_1)\mapsto (y(t),\theta(t))$ 
is continuous with respect to the $H^{s+1}\times H^{s+\frac12}$ topology for $t\in [0,T)$.
\end{theorem}

Let us make some comments on the main theorem.

%

\medskip

\noindent\textbf{Explicit condition on the radius $R$.} 
The radius of the ball of initial data for which we are able to prove 
the existence of classical solutions can be explicitly computed. 
Indeed this is determined by the following condition: 
there exists a constant $\mathtt{c}_3>0$ such that
\begin{equation}\label{explicitR}
c(x)+\partial_{h_6} F_2(h_1, h_2, h_3, h_4, h_5, h_6)>\mathtt{c}_3\,, 
\qquad \forall x\in \T,\,\,\forall |h_i|< R,\,\,i=1, \dots, 6\,.
\end{equation}
In a regime close to equilibrium, namely for \emph{small enough} initial data, 
the above conditions is verified. 
Since we are considering a model that describes the evolution 
of a bridge in a small torsional regime we expect this to be the typical situation in applications. 

Since we consider quadratic nonlinearities one cannot expect a \emph{global} ellipticity condition 
as in \cite{KVP,FI2018, FIJMPA}. This is the reason why we have to restrict to \emph{small} initial data.
We decided to consider quadratic nonlinearity which arises more naturally from bridges models (see for instance \cite{ArioliGazzola2017}).

\medskip

\noindent\textbf{Regularity of the solutions.}  
Our result provides the existence of \emph{classical} solutions. 
We tried to optimise to the best  (following our method)
the amount of Sobolev regularity required on the initial conditions.
However we do not know if such threshold is optimal.
Regarding the regularity of the solution map we show that it is continuous. We point out that, in the quasilinear context, one cannot expect high regularity of the flow map, as in the case of semilinear PDEs (for which we refer to \cite{cazenaveBook}). 
We mention \cite{MolSautTzv2001} and \cite{Tzvetkov2007}
for more details on this issue.

%
%
%
%

\medskip


\medskip

\noindent\textbf{Damping and forcing terms.} The damping terms 
$\alpha y_t$, $\beta \theta_t$ and forcing terms $\delta f(t)$, $\gamma g(t)$ 
are relevant in engineering applications.  
However our method is not affected by the presence of such terms. 
Then our result holds even in the case that the suspension bridge 
is considered as an isolated system $\alpha=\beta=\gamma=\delta=0$.  
The importance of this regime has been highlighted by Irvine, see \cite{Irvine}.

\medskip

\noindent\textbf{Boundary conditions.} 
%
In engineering applications to \emph{hinged} decks, 
one should study equations \eqref{sistemaponte}
on the interval $x\in [0,\pi]$ (or more in general $[0,L]$)
under the following boundary conditions:
\begin{equation}\label{conditionDirich}
y(t, 0)=y(t, \pi)=y_{xx}(t, 0)=y_{xx}(t, \pi)=0\,, 
\quad 
\theta(t, 0)=\theta(t, \pi)=0\,, \qquad x\in [0, \pi]\,.
\end{equation}
First of all we remark that the boundary conditions \eqref{conditionDirich} 
are equivalent to the Dirichlet boundary conditions
\begin{equation}\label{conditionDirich2}
y(t, 0)=y(t, \pi)=0\,, \quad \theta(t, 0)=\theta(t, \pi)=0\,, \qquad x\in [0, \pi]\,,
\end{equation}
for functions $y(x)$, $\theta(x)$, $x\in[0,\pi]$ admitting 
Fourier series expansion 
\begin{equation}\label{fousin}
y(t, x)=\sqrt{\frac{2}{\pi}} \sum_{n\geq 1} a_n \sin(n x)\,, 
\qquad 
\theta(t, x)=\sqrt{\frac{2}{\pi}} \sum_{n\geq 1} b_n \sin(n x)\,,
\end{equation}
provided that they are \emph{sufficiently regular} in $x$.
This means that, to get regular solutions of 
\eqref{sistemaponte}-\eqref{conditionDirich}
one can look for regular solutions of \eqref{sistemaponte}-\eqref{conditionDirich2}
of the form \eqref{fousin}.
The local well posedness result for the Dirichlet problem 
\eqref{sistemaponte}-\eqref{conditionDirich2} shall follow 
as a corollary of Theorem \ref{thm:main1}, 
under appropriate parity assumptions.\footnote{We remark that these assumptions arise naturally when 
one wants to perform a \emph{modal analysis}. 
For instance it is satisfied for models involving non-local nonlinearities \emph{\`a la} Kirchoff.} 
In other words, instead of studying directly the hinged problem \eqref{sistemaponte}-\eqref{conditionDirich},
we consider \eqref{sistemaponte} on $\mathbb{T}$ and look for \emph{odd} in $x$ regular solutions.
%
%
%
%
%
This is the content of the following result.
\begin{cor}[\bf{Dirichlet/hinged boundary conditions}]\label{coro1}
Under Hypothesis \ref{hyp1} assume that
the operators $\mathcal{B},\mathcal{W}$ preserve the subspaces of 
\emph{even}/\emph{odd} functions in $x\in \mathbb{T}$ \footnote{
$\mathcal{B},\mathcal{W}$ are said \emph{parity-preserving}.} and that
 \[
 \begin{aligned}
&F_i(x,-h_1,h_2,-h_3,-h_4,h_5,-h_6)
=-F_i(x,h_1,h_2,h_3,h_4,h_5,h_6)\,, 
\\& \forall x\in \T, h_i\in \R,\,\,i=1, \dots, 6\,.
\end{aligned}
\]
Then for any initial condition $(y_0,\theta_0)\in H^{s+1}\times H^{s+\frac12}$
\emph{odd} in $x\in \T$ the thesis of Theorem \ref{thm:main1}
holds and in addition the solution $(\theta(t),y(t))$
is odd with respect to $x$.
\end{cor}

\begin{proof}
It follows by Theorem \ref{thm:main1} and the fact that the above
assumptions imply that the vector 
field in the right hand side of \eqref{sistemaponte}
preserves the subspace of odd functions of $x\in \T$.
\end{proof}

\medskip

\noindent\textbf{The energy method.}
Since the equations in system \eqref{sistemaponte} are quasilinear we cannot expect to 
obtain the existence of low regular solutions. 
Even the presence of non constant coefficients in the linear terms 
(see \eqref{opastratti}) prevents the use of classical approaches
based on standard fixed points arguments.
In the quasilinear context, the usual strategy is to provide more refined \emph{a priori} estimates on the solutions. 
We mention for instance the papers \cite{KPV2, KVP, MMT3}
for the quasilinear Schr\"odinger equation on the Euclidean space.
In these latter papers the dispersive properties of the linear flow have been somehow exploited.

On compact manifolds, where no dispersive effects are available, 
the situation is known to be more involved.
For instance there are 
very interesting examples given by Christ in \cite{Christ} of non 
Hamiltonian equations which are ill posed on the torus $\mathbb{T}$ and well posed on $\mathbb{R}$.
We mention 
(without trying to be exhaustive) the well posedness results 
\cite{FI2018, FIJMPA} for the Schr\"odinger equation on tori,  \cite{iandolikdv} for the Kortweg-deVries
(see also \cite{Berti-Maspero-Murgante:EK} for a different model).
In this paper we mostly follow the approach of \cite{FIJMPA}, based on a 
quasilinear iterative scheme  \emph{\`a la Kato} \cite{Kato:spectral}
starting from a \emph{paralinearized} version (see M\'etivier \cite{Metivier2008}) 
of the system \eqref{sistemaponte}.
The key point 
is to provide suitable a priori energy estimates
on the solutions of a linear paradifferential systems 
arising form \eqref{sistemaponte}. 
The extra difficulty, compared with respect to \cite{FIJMPA} is due to the fact that 
 \eqref{sistemaponte} is a system of \emph{coupled} equations. In order to apply the 
method developed in \cite{FIJMPA} we first rewrite  the system in complex coordinate \eqref{inverse}.
In this set of coordinates the coupling terms appear at order $1/2$ (see the {paradifferential} operator 
$\mathfrak{B}$ in \eqref{ponteparato}-\eqref{matricitotali}). 
This coupling at \emph{unbounded} orders is due to the presence 
of the second order derivatives $y_{xx}$, $\theta_{xx}$ in the nonlinearities in \eqref{sistemaponte}.
To obtain energy estimates one has to \emph{decouple},  the system, at least, at positive orders.
This is achieved through suitable paradifferential changes of coordinates 
performed in Sections \ref{sec:modene}-\ref{sec:linear}.

\subsection{A model for suspension bridges}\label{sec:derivation}
%

%


There is a wide literature concerning the derivation of mathematical models for suspension bridges and the study of its dynamics through modal analysis, mainly by using numerical simulations. {The first study on suspension bridges dates back to Navier \cite{Navier}. In the recent years, the collapse of some suspension bridges due to aerodynamic instability combined with uncontrolled oscillations (see for a wide overview on the topic \cite{Gazzola:book}), 
raised fundamental questions in both mathematicians and engineers.
In this context, a great interest has been addressed to the explanation of the famous collapse of the Tacoma Bridge in 1940. 
Still today, there is no common agreement on the explanation of the causes of the collapse.

 Vortex shedding, flutter theory or parametric resonance are some of the aeroelastic effects that were believed to be at the origin of these phenomena \cite{Butikov, Herm, JENKINS2013167}.

A more recent research line \cite{ArioliGazzola2015, ArioliGazzola2017, BerchioGazzola2015} explains the collapse as a phenomenon of structural instability: exchanges of energy among internal (longitudinal and torsional) modes may cause catastrophic effects. In this context, usually, one considers the suspension bridges as isolated systems, ignoring damping terms, as suggested by Irvine in \cite{Irvine}.

%
 

The possibility of exchanges of energy among the two components of the deck suggests that the bridge behaves \emph{nonlinearly}, which is
nowadays commonly agreed,  see for instance \cite{Brownjohn1994, Gazzola2013, Lacarbonara2013, McKenna-Walter1987}
and references therein.  The refined model proposed in \cite{ArioliGazzola2017} takes into account the nonlinear restoring 
action of the cables-hangers system on the deck; 
considering \emph{deformable cables} and \emph{hangers} 
makes the descriptive system of equations \emph{quasilinear} (in the sense explained above).  {As we said, this complicates significantly the study of the well posedness of the Cauchy problem, and, at the best of our knowledge, Theorem \ref{thm:main1} is the first result of well posedness for a quasilinear system of PDEs modelling suspension bridges.}.  

The aim of the present paper is to propose a general method, based on microlocal analysis and paradifferential calculus technique, that enable one to treat general system as, for instance, the one considered in
\cite{ArioliGazzola2017}. }
Let us show that the system of equations derived in that paper 
has the form \eqref{sistemaponte}. 
We write in detail the linear part of equations $(6)$-$(7)$ in \cite{ArioliGazzola2017}, 
while we leave the nonlinear terms in implicit form. 
We have
\begin{equation}\label{systemAG}
\begin{aligned}
(M+2 \,m \, \xi(x)) \,y_{tt}&=-EI \,y_{xxxx}+\partial_x \Big(\frac{2\,H_0}{\xi(x)^2}\,y_x \Big)
+\mathcal{N}^{(y)}(y,y_x, y_{xx}, \theta, \theta_x, \theta_{xx})\,,
\\
\Big(\frac{M}{3}+2\,m\,\xi(x) \Big) \,\theta_{tt}&=\frac{GK}{\ell^2} \,\theta_{xx}
+\partial_x\Big(\frac{2\,H_0}{\xi(x)^2}\,\theta_{x}\Big)
+\mathcal{N}^{(\theta)}(y,y_x, y_{xx}, \theta, \theta_x, \theta_{xx})\,,
\end{aligned}
\end{equation}
where
\begin{itemize}
\item[$\bullet$] $M$ is the linear density of the mass of the deck;
 \item[$\bullet$] $m$ is the linear density of the mass of the cable;
 \item[$\bullet$] $\xi=\xi(x)$ represents the local length of the cable 
 at rest and $H_0\,\xi$ is the tension of the cable at rest;
 \item[$\bullet$] $E$ is the Young modulus and $I$ is the linear density 
 of the moment of inertia of the cross section;
 \item[$\bullet$] $G$ is the shear modulus of the deck;
 \item[$\bullet$] $K$ is the torsional constant of the deck.
  \end{itemize}
Dividing by $M+2 \,m \, \xi$ the equation for the $y$ variable, by $M/3+2 \,m \, \xi$ 
the equation for the $\theta$ variable and 
by scaling time the system \eqref{systemAG} takes the form of \eqref{sistemaponte}. 
We observe that the coefficients $b(x)$ and $c(x)$ in \eqref{opastratti} 
are given by
\[
b(x)=\frac{EI}{M+2m\,\xi}>0, \qquad
c(x)=\frac{3 G K }{\ell^2 (M+6 m\,\xi)}+\frac{6 H_0}{ \xi^2 (M+6 m\,\xi)}>0\,.
\]
Hence the Hypothesis \ref{hyp1} is satisfied. We observe that the extra assumption required in Corollary \ref{coro1} is not satisfied by the specific nonlinearities $\mathcal{N}^{(y)}, \mathcal{N}^{(\theta)}$ considered in \cite{ArioliGazzola2017}. However we believe that this assumption is quite natural, especially if one wants to perform an efficient modal analysis.


\section{Functional setting and Paradifferential calculus}
Let $s\in \R$. 
We expand a function $ u(x) $, $x\in \mathbb{T}$, 
 in Fourier series as 
\begin{equation}\label{complex-uU}
u(x) =
\sum_{n \in \Z } \hat{u}(n)\,e^{\ii n x } \, , 
\qquad 
\hat{u}(n) := \frac{1}{2\pi} \int_{\mathbb{T}} u(x) e^{-\ii n  x } \, dx \, .
\end{equation}
We introduce the Sobolev spaces  
\begin{equation}\label{spazioHHss}
H^{s}=H^{s}(\T;\C):=\Big\{ u(x)\in L^{2}(\T;\C)\,:\, 
\| u\|_{H^{s}}^2:=\sum_{j\in \mathbb{Z}} |\hat{u}(j)|^2 \langle j \rangle^{2s}<+\infty \Big\}\,
\end{equation}
and the Japanese bracket is defined by $\langle j \rangle:=\sqrt{1+j^2}$.
We introduce the Fourier multiplier $\langle D\rangle$ defined by linearity as
\begin{equation}\label{Foulangle}
\langle D\rangle e^{\ii jx}=\langle j\rangle e^{\ii jx}\,,\quad \forall\, j\in\Z\,.
\end{equation}
Notice that the norm $\|\cdot\|_{H^{s}}$ can be written as
\[
\|u\|_{H^s}^{2}=(\langle D\rangle^{2s} u, u)_{L^{2}}\,,
\]
where
\begin{equation}\label{scalarproScalare}
(u,v)_{L^{2}}=\frac{1}{2\pi}\int_{\mathbb{T}} u\cdot \bar{v}dx\,,\quad \forall u,v\in L^{2}(\mathbb{T};\C)\,,
\end{equation}
is the standard complex $L^2$-scalar product.
We introduce also the product  spaces 
\begin{equation}\label{RealSobolev}
\mathcal{H}^s=\mathcal{H}^s(\T; \mathbb{C})=\{ U=(u^+, u^-)\in 
H^s\times H^s : u^+=\overline{u^-} \}\,.
\end{equation}
With abuse of notation we shall denote by $\|\cdot\|_{H^{s}}$ the 
natural norm of the product space $\mathcal{H}^{s}$.
On the space $ \mathcal{H}^s$
we naturally extends the scalar product \eqref{scalarproScalare} as
\begin{equation}\label{scalarprod2x2}
(Z,W)_{L^{2}\times L^{2}}:={\rm Re}(z,w)_{L^{2}}=\frac{1}{4\pi}\int_{\T} z\cdot \bar{w}+\bar{z}\cdot w dx\,,
\qquad Z=\vect{z}{\bar{z}}\,,\; W=\vect{w}{\bar{w}}\in \mathcal{H}^{0}\,.
\end{equation}
We define the matrices of operators
\[
\mathcal{D}:=\left(\begin{matrix}\langle D\rangle & 0 \\ 0 & \langle D\rangle \end{matrix}\right)\,,
\qquad 
\mathfrak{D}:=\left(\begin{matrix}\mathcal{D} & 0 \\ 0 & \mathcal{D} \end{matrix}\right)\,.
\]
We notice that 
\[
\|Z\|_{H^{s}}^{2}:=(\mathcal{D}^{2s}Z,Z)_{L^{2}\times L^{2}}\,,\qquad \forall\,Z\in \mathcal{H}^{s}\,,
\]
and (again with abuse of notation) we set
\[
\|V\|_{H^{s}}^{2}:=\langle\mathfrak{D}^{2s}V,V \rangle:=\|Z\|_{H^{s}}^2+\|W\|_{H^{s}}^{2}
\,,
\qquad \forall\, 
V=\vect{Z}{W}\in \mathcal{H}^{s}\times\mathcal{H}^s\,,
\]
where
\begin{equation}\label{scalarprod4x4}
\langle V_1,V_2\rangle:=(Z_1,Z_2)_{L^{2}\times L^{2}}+(W_1,W_2)_{L^{2}\times L^{2}}\,,
\;\;
V_1=\vect{Z_1}{W_1}\,,V_2=\vect{Z_2}{W_2}\in \mathcal{H}^{0}\times\mathcal{H}^0\,.
\end{equation}

\medskip

\noindent
\textbf{Notation}.
We use the notation $f\lesssim g$ to denote inequalities 
$f\le C g$ for some $C>0$. 

\noindent
If we want to highlight the 
dependence of the previous constants on some parameter $s$ 
(as the index of the Sobolev regularity) then we write $f \lesssim_s g$.

\medskip

\noindent
We recall the following classical tame estimates on Sobolev spaces: for all $s\geq s_0>1/2$
\begin{itemize}
\item[$\bullet$] \emph{Tame product}:
\begin{equation}\label{tame}
\|uv\|_{H^{s}}\lesssim\|u\|_{H^{s}}\|v\|_{H^{s_0}}+\|u\|_{H^{s_0}}\|u\|_{H^{s}}\,.
\end{equation} 
\item[$\bullet$] \emph{Interpolation estimate}:
\begin{equation}\label{intorpolostima}
\|u\|_{H^{s}}\leq \|u\|^{\theta}_{H^{s_1}}\|u\|^{1-\theta}_{H^{s_2}}\,,
\qquad \theta\in[0,1]\,,\;\;0\leq s_1\leq s_2\,,\quad s=\theta s_1+(1-\theta)s_2\,.
\end{equation}
\end{itemize}

\noindent
We now recall some results concerning the paradifferential calculus. 
We  follow \cite{Berti-Maspero-Murgante:EK}. 

\begin{definition}\label{def:simbolini}
Given $m,s\in\mathbb{R}$ we denote by $\Gamma^m_s$ 
the space of functions $a(x,\xi)$ defined on $\T\times \R$ 
with values in $\C$, which are $C^{\infty}$ 
with respect to the variable $\xi\in\R$ and such that for any $\beta\in \N\cup\{0\}$,
there exists a constant $C_{\beta}>0$ such that 
\begin{equation}\label{stima-simbolo}
\|\partial_{\xi}^{\beta} a(\cdot,\xi)\|_{H^{s}}
\leq C_{\beta}\langle\xi\rangle^{m-\beta}\,, 
\quad \forall \xi\in\R.
\end{equation}
The elements of $\Gamma_s^m$ are called \emph{symbols} of order $m$.
\end{definition}
We endow the space $\Gamma^m_{s}$ with the family of {semi}-norms 
\begin{equation}\label{seminorme}
|a|_{m,s,n}:=\max_{\beta\leq n}\, 
\sup_{\xi\in\R}\|\langle\xi\rangle^{\beta-m}\partial^\beta_\xi a(\cdot,\xi)\|_{H^{s}}\,.
\end{equation}
Consider a function $\chi\in C^{\infty}(\R,[0,1])$ such that 
\[
\chi(\xi)=\begin{cases}
1 \qquad \mathrm{if}\,\,|\xi|\leq 1.1,\\
0 \qquad \mathrm{if}\,\,|\xi|\geq 1.9.
\end{cases}
\]
Let $\epsilon\in(0,1)$ and define 
$\chi_{\epsilon}(\xi):=\chi(\xi/\epsilon).$ 

\begin{definition}
Given a symbol $a(x,\x)$ in $\Gamma^m_s$
we define its Weyl and Bony-Weyl
quantization respectively as 
\begin{equation}\label{quantiWeylcl}
T_{a}h:=\opw(a(x,\xi))h:=\sum_{j\in \mathbb{Z}}e^{\ii j x}
\sum_{k\in\mathbb{Z}}
\widehat{a}\left(j-k,\frac{j+k}{2}\right)\widehat{h}(k),
\end{equation}
\begin{equation}\label{quantiWeyl}
\opbw(a(x,\xi))h:=\sum_{j\in \mathbb{Z}}e^{\ii j x}
\sum_{k\in\mathbb{Z}}
\chi_{\epsilon}\left(\frac{|j-k|}{\langle j+k\rangle}\right)
\widehat{a}\left(j-k,\frac{j+k}{2}\right)\widehat{h}(k).
\end{equation}
\end{definition}
In the following we list a series of technical lemmata that we repeatedly use along the paper.

The following is Theorem $2.4$ in \cite{Berti-Maspero-Murgante:EK}  and 
concerns the action of paradifferential operators on Sobolev spaces.
\begin{theorem}{\bf (Action on Sobolev spaces).}\label{azione}
Let $s_0>1/2$, $m\in\R$ and $a\in\Gamma^m_{s_0}$. Then $\opbw(a)$ 
extends as a bounded operator from ${H}^{s}(\T;\C)$ to ${H}^{s-m}(\T;\C)$, 
for all $s\in\R$, with the following estimate  
\begin{equation}\label{ac1}
\|\opbw(a)u\|_{H^{s-m}}\lesssim |a|_{m,s_0,4}\|u\|_{H^s} \qquad \forall u\in{H}^s\,.
\end{equation}
Moreover for any $\rho\geq 0$ we have 
\begin{equation}\label{ac2}
\|\opbw(a)u\|_{H^{s-m-\rho}}\lesssim |a|_{m,s_0-\rho,4}\|u\|_{H^s} \qquad {u}\in {H}^s\,.
\end{equation}
\end{theorem}

\begin{remark}\label{rmk:bonynonbony}
Given $a\in \Gamma_{s}^{m}$, 
we remark that the operator
\[
R:=\opw(a(x,\x))-\opbw(a(x,\x))
\]
is bounded from $H^{s}\to H^{s+\rho}$ for any $\rho>0$.
Moreover for any $s\in\R$ we have
\[
\|R\|_{\mathcal{L}(H^s;H^{s+\rho})}\lesssim_{s,\rho}|a|_{m,s+\rho,4+\rho}\,.
\]

\end{remark}

\begin{definition}
Let $\rho\in (0, 2]$. Given two symbols $a\in \Gamma^m_{s_0+\rho}$ and $b\in \Gamma^{ m'}_{s_0+\rho}$, we define
\begin{equation}\label{cancelletto}
a\#_{\rho} b= \begin{cases}
ab \quad \qquad \quad \quad\,\, \rho\in(0,1]\\
ab+\frac{1}{2\ii}\{a,b\}\quad \rho\in (1,2],\\
\end{cases}
\end{equation}
where we denoted by $\{a,b\}:=\partial_{\xi}a\partial_xb-\partial_xa\partial_{\xi}b$ the Poisson's bracket between symbols.
\end{definition}
\begin{remark}\label{simmetrie}
We observe that 
\[
ab\in\Gamma^{m+m'}_{s_0+\rho}, \quad \{a,b\}\in\Gamma^{m+m'-1}_{s_0+\rho-1}.
\] 
Moreover $\{a,b\}=-\{b,a\}$. 
\end{remark}

The following is Theorem $2.5$ in \cite{Berti-Maspero-Murgante:EK}. This  result concerns the symbolic calculus for the composition of Bony-Weyl paradifferential operators. 
\begin{theorem}{\bf (Composition).}\label{compo}
Let $s_0>1/2$, $m,m'\in\mathbb{R}$, $\rho\in(0,2]$ and 
$a\in\Gamma^m_{s_0+\rho}$, $b\in\Gamma^{m'}_{s_0+\rho}$. We have 
\[
\opbw(a)\circ\opbw(b)=\opbw(a\#_{\rho}b)+R^{c}_{\rho}(a,b)\,,
\]
 where the linear operator $R^{-\rho}$ maps ${H}^s(\T)$ into ${H}^{s+\rho-m-m'}$, 
 for any $s\in\R$.
 
 Moreover it satisfies the following estimate, for all $u\in {H}^s$,
\begin{equation}\label{resto}
\|R^{c}_{\rho}(a,b)u\|_{{H}^{s-(m+m')+\rho}}
\lesssim_s 
(|a|_{m,s_0+\rho,N}|b|_{m',s_0,N}+|a|_{m,s_0,N}|b|_{m',s_0+\rho,N})\|u\|_{{H}^s},
\end{equation}
where $N\geq 7$.\end{theorem}

The next result is Lemma $2.7$ in \cite{Berti-Maspero-Murgante:EK}.
\begin{lemma}{\bf (Paraproduct).} \label{ParaMax} 
Fix $s_0>1/2$ and let $f\in H^s$ and $g\in H^r$ with $s + r \ge 0$. Then
\begin{equation}
\label{eq: paraproductMax}
fg=\opbw(f)g+\opbw({g})f+{R}^p(f,g)\,,
\end{equation}
where the bilinear operator $R: H^s \times H^r\to H^{s + r - s_0}$ satisfies the estimate
\begin{equation}
\label{eq: paraproductMax2}
\|{R}^p(f,g)\|_{H^{s+r - s_0}} \lesssim_s \|f\|_{H^s} \|g\|_{H^r}\,.
\end{equation}
\end{lemma}

\subsection{Pseudo-differential structure of the linearizad equations}
Consider 
the operators $\mathcal{B},\mathcal{W}$ in \eqref{opastratti}.
In this subsection we shall give a more explicit expression of the 
highest order symbols of the operators $\mathcal{B},\mathcal{W}$.
More precisely we have the following.
\begin{lemma}\label{lem:formexplicBW}
Let $s\geq 0$.
There exists two operators 
$\widetilde{\mathtt{R}}_i\in \mathcal{L}(H^{s}; H^{s-\frac{2}{i}})$,
$i=1,2$, such that 
\begin{equation}\label{opside1}
\begin{aligned}
\mathcal{B}&=\opbw(-\x^4b(x)-2\ii\x^3 b_{x}(x))+\widetilde{\mathtt{R}}_1\,,
\\
\mathcal{W}&=\opbw(-\x^2c(x))+\widetilde{\mathtt{R}}_2\,.
\end{aligned}
\end{equation}
Moreover they satisfy the estimates 
\begin{equation}\label{latte2}
\|\langle D\rangle^{-\frac{1}{\ii}}\widetilde{\mathtt{R}}_i
\langle D\rangle^{-\frac{1}{i}}\|_{\mathcal{L}(H^{s};H^{s})}
\lesssim_s 1\,,\quad i=1,2\,.
\end{equation}
\end{lemma}

\begin{proof}
First of all, in view of Remark
\ref{rmk:bonynonbony} and recalling \eqref{quantiWeylcl}-\eqref{quantiWeyl},
we rewrite the multiplication operators by $b(x)=\opw(b(x))$ and $c(x)=\opw(c(x))$ as
\begin{equation}\label{latte}
\begin{aligned}
\opw(b(x))&=\opbw(b(x))+\mathtt{Q}_1\,,
\\
\opw(c(x))&=\opbw(c(x))+\mathtt{Q}_2\,, 
\end{aligned}
\end{equation}
for some remainders
satisfying
\[
\begin{aligned}
\|{\mathtt{Q}_{1}}\|_{\mathcal{L}(H^{s};H^{s+\rho})}&\lesssim_{s,\rho}
\|b\|_{H^{s+\rho}}\,,
\qquad
\|\mathtt{Q}_{2}\|_{\mathcal{L}(H^{s};H^{s+\rho})}\lesssim_{s,\rho}
\|c\|_{s+\rho}, \qquad \forall \rho\geq 0,\,\,\forall s\in \R\,.
\end{aligned}
\]
Recalling that $\pa_{x}=\opbw(\ii\x)$ and using the composition Theorem \ref{compo} 
we deduce
that
\[
\begin{aligned}
- \opbw(b(x))  \pa_{x}^4 &=
- \opbw(b(x))  \circ\opbw(\x^4)
\\&
-\opbw\Big(
\x^{4}b(x)+\frac{1}{2\ii}\{b(x),\x^4\}\Big)+Q_1
\\&
=-\opbw(\x^4b(x)+2\ii b_x(x)\x^3)+Q_1
\\
\opbw(c(x))\pa_{x}^2&=\opbw(-c(x)|\x|^2)+Q_2
\end{aligned}
\]
for some remainders $Q_{1},Q_2$ satisfying \eqref{latte2}.
Recalling \eqref{latte} we obtained  \eqref{opside1}.
\end{proof}

\section{Paralinearization of the system}\label{sec:parasez}

In this section we transform the system \eqref{sistemaponte} 
using the paraproduct formula. 
Recalling the Fourier multiplier in \eqref{Foulangle} and
according to \eqref{quantiWeyl}, we shall write
$\langle D\rangle:=\opbw(\langle\x\rangle)$.
We set
\begin{equation}\label{inverse}
\begin{aligned}
y={\sqrt{2}}\langle D\rangle^{-1}\Re(z)\,,
\quad 
\theta={\sqrt{2}}\langle D\rangle^{-1/2}\Re(w)\,,\quad
V:=\left(\begin{matrix}Z \\ W\end{matrix}\right)\,,\quad Z=\vect{z}{\bar{z}}\,,\;\;
W=\vect{w}{\bar{w}}
\end{aligned}
\end{equation}
and we  define the functions
\begin{equation}\label{simbog}
\begin{aligned}
a(x)&:=\tfrac{1}{2}(b(x)-1)\,,
\qquad 
d(x):=\tfrac{1}{2}(c(x)-1)\,,
\\
\mathtt{g}_{1,w}(V;x)
&:=\tfrac{1}{2}(\pa_{\theta_{xx}}F_2)(x,y,y_x,y_{xx},\theta,\theta_x,\theta_{xx})\,,
\\
\mathtt{g}_{\frac{1}{2},b}(V;x)
&:=\tfrac{1}{2}(\pa_{\theta_{xx}}F_1)(x,y,y_x,y_{xx},\theta,\theta_x,\theta_{xx})\,,
\\
\mathtt{g}_{\frac{1}{2},w}(V;x)
&:=\tfrac{1}{2}(\pa_{y_{xx}}F_2)(x,y,y_x,y_{xx},\theta,\theta_x,\theta_{xx})\,,
\end{aligned}
\end{equation}
and  the matrices
\[
E:=\sm{1}{0}{0}{-1}\,,\qquad {\bf U} :=\sm{1}{1}{1}{1}\,,
\qquad \Id_{\C^2}:=\sm{1}{0}{0}{1}\,.
\]
It is also convenient to set the following notations for 
symbols:
 \begin{align}
 A_b(x,\xi)&:=\Big((\Id_{\C^2}+{\bf U} a(x))\,\xi^2+2\ii {\bf U} a_x(x)\xi\Big)\label{matriceregobeam}
 \\
 A_w(x,\xi)&:= A_w(V;x,\xi)
:=\Big(\Id_{\C^2}+{\bf U} a_{w}(V;x)\Big)|\x|\label{matriceregowave}\,,
 \\
 a_w(x)&:=a_w(V;x):=d(x) +\mathtt{g}_{1,w}(V;x)\,,
 \label{simbolowave}
 \\
  B_b(x,\xi)&:=B_b(V;x,\xi):={\bf U} \mathtt{g}_{\frac12,b}(V;x)\langle\xi\rangle^{-\frac{3}{2}}\xi^2
  \label{simbolobeamoff}
  \\
    B_w(x,\xi)&:=B_w(V;x,\xi):={\bf U} \mathtt{g}_{\frac12,w}(V;x)\langle\xi\rangle^{-\frac32}\xi^2\,.
  \label{simbolowaveoff}
 \end{align}

 \textbf{Notation.} From now on when we write $\lesssim_s$ we mean that there are constants that could depend on the $H^s$ Sobolev norm of the coefficients $b(x)$ and $c(x)$, and so also of the functions $a(x)$ and $d(x)$.
 
 \medskip

We fix the parameters
 \begin{equation}\label{parametrifissati}
 s_0>\frac{1}{2}\,,\quad s_1:= s_0+\frac{3}{2}\,,\qquad s_2:= s_1+2\,,\quad \mathfrak{s}:= s_2+1\,.
 \end{equation}
 \begin{itemize}
 \item[$\bullet$] The parameter $s_0$ denotes the regularity threshold for which $H^{s_0}$
 is an algebra.
 
 \item[$\bullet$] The parameter $s_1$ denotes the regularity of the coefficients of the paralinearized system and the remainder of the paralinearization procedure will be bounded on $H^s$ spaces with $s\geq s_1+1$. The index $s_1$ depends on the order of derivatives appearing in the nonlinear terms $F_1, F_2$ in \eqref{sistemaponte}. 

 \item[$\bullet$] The parameter $s_2$ denotes the minimal regularity
 in which we are able to provide energy estimates for the system.
 \item[$\bullet$] The parameter $\mathfrak{s}$ denotes the minimal regularity in which we 
prove the existence results in the \emph{real} coordinates $(y, \theta)$.

 \end{itemize}

We prove the following.
\begin{proposition}\label{paraparaProp}
Let $s\geq s_1+1$.
Under the Hypothesis \ref{hyp1} one has that 
system \eqref{sistemaponte}
is equivalent to
\begin{equation}\label{ponteparato}
\pa_{t}V=\mathfrak{A}(V)V+\mathfrak{B}(V)V+\mathtt{R}V+\mathcal{R}(V)+G(t)\,,
\end{equation}
where $V$ is as in \eqref{inverse} and

\noindent
$(i)$ 
$\mathcal{R} :\mathcal{H}^{s}\times \mathcal{H}^{s}\to \mathcal{H}^s\times\mathcal{H}^s$
is bilinear and
\begin{equation}\label{stimaRRR}
\|\mathcal{R}(V)\|_{H^{s}}\lesssim_{s}\|V\|_{H^s}\|V\|_{H^{{s}_1+1}}\,,
\quad \forall\,V\in \mathcal{H}^{s}\times\mathcal{H}^{s}\,.
\end{equation}

\noindent
$(ii)$ The linear operator 
$\mathtt{R}\in \mathcal{L}(\mathcal{H}^{s}\times\mathcal{H}^{s};
\mathcal{H}^{s}\times\mathcal{H}^{s})$
satisfies the estimate 
\begin{align}
\|\mathtt{R} V\|_{H^s}\lesssim_s& \|V\|_{H^s}\,,\quad \forall\, V\in \mathcal{H}^{s}\times\mathcal{H}^{s}\,.
\label{restilineari}
\end{align}
Moreover it has the form
\[
\mathtt{R}=\left(\begin{matrix}\mathtt{R}_1 & 0 \\0& \mathtt{R_2}\end{matrix}\right)\,,
\qquad \mathtt{R}_i\in \mathcal{L}(\mathcal{H}^{s};\mathcal{H}^s)\,,\; i=1,2\,.
\]

\noindent
$(iii)$ We have that
\[
\begin{aligned} 
&\mathcal{H}^{{s}_1}\times\mathcal{H}^{{s}_1}\ni V
\mapsto
\mathfrak{A}(V)[\cdot]\in 
\mathcal{L}(\mathcal{H}^{s}\times\mathcal{H}^{s};\mathcal{H}^{s-2}\times\mathcal{H}^{s-1})
\\
&\mathcal{H}^{{s}_1}\times\mathcal{H}^{{s}_1}\ni V
\mapsto
\mathfrak{B}(V)[\cdot]\in 
\mathcal{L}(\mathcal{H}^{s}\times\mathcal{H}^{s};\mathcal{H}^{s-\frac{1}{2}}\times\mathcal{H}^{s-\frac{1}{2}})
\end{aligned}
\]
and the operators $\mathfrak{A}(V)[\cdot], \mathfrak{B}(V)[\cdot ]$ have the form 
\begin{equation}\label{matricitotali}
\begin{aligned}
\mathfrak{A}(V)[\cdot]&=\left(
\begin{matrix}
-\ii E\opbw(A_b(x,\x)) & 0 \\ 0 &-\ii E\opbw(A_{w}(V;x,\x))
\end{matrix}
\right)\,,
\\
\mathfrak{B}(V)[\cdot]&=\left(
\begin{matrix}
0  & -\ii E\opbw(B_b(V;x,\x))  \\ -\ii E\opbw(B_{w}(V;x,\x))& 0
\end{matrix}
\right)\,.
\end{aligned}
\end{equation}

\noindent
$(iv)$ The symbol in \eqref{matriceregobeam} satisfies
(recall \eqref{seminorme})
\begin{equation}\label{opside2}
|A_{b}|_{2,s,\alpha}\lesssim_{s} 1\,,
\quad 
\forall\,s\in\R\,,\alpha\in\N\,,
\end{equation}
while the symbols in \eqref{matriceregowave}-\eqref{simbolobeamoff}-\eqref{simbolowaveoff} satisfy 
{for all $\alpha\in\mathbb{N}$, for all $V, V_1, V_2\in \mathcal{H}^s\times \mathcal{H}^s$,}
\begin{align}
&|B_{b}(V;x,\xi)|_{\frac12,p,\alpha}+|B_{w}(V;x,\xi)|_{\frac12,p,\alpha}\lesssim_s 
\|V\|_{p+\frac32}\,, \qquad  s_0\leq p\leq s-3/2\,,
\label{figaro3}
\\
&|A_{w}(V;x,\xi)|_{1,p,\alpha}\lesssim \|V\|_{p+\frac32}\,,
\qquad  \qquad \qquad \qquad \qquad s_0\leq p\leq s-3/2\,,\label{figaro4}\\
&|A_{w}(V_1;x,\xi)-A_{w}(V_2;x,\xi)|_{1,p,\alpha}\lesssim \|V_1-V_2\|_{p+\frac32}\,,\quad s_0\leq p\leq s-3/2,\label{figaro55}\\
&|B_{*}(V_1;x,\xi)-B_{*}(V_2;x,\xi)|_{\frac12,p,\alpha}\lesssim \|V_1-V_2\|_{p+\frac32},\quad  *\in\{b,w\},\,s_0\leq p\leq s-3/2\label{figaro66}\,.
\end{align}

\noindent
$(v)$ The forcing term $G(t)$ has the form
\[
G(t):=\vect{G_{b}(t)}{G_{w}(t)}\,,
\qquad G_{b}(t):=\frac{\ii\gamma}{\sqrt{2}}\vect{\langle D\rangle^{-1}\mathtt{f}_{b}(t)}{\langle D\rangle^{-1}\mathtt{f}_{b}(t)}\,,
\quad
G_{w}(t):=\frac{\ii\delta}{\sqrt{2}}\vect{\langle D\rangle^{-1/2}\mathtt{f}_{w}(t)}{\langle D\rangle^{-1/2}\mathtt{f}_{w}(t)}\,.
\]
\end{proposition}

\begin{proof}
First of all, in the complex coordinates \eqref{inverse} 
the system \eqref{sistemaponte} reads as
\begin{align}
\pa_{t}z&= \langle D\rangle^2\Im(z)
+\ii\langle D\rangle^{-1}\mathcal{B}\langle D\rangle^{-1}\Re(z)
+\ii\alpha{\rm Im}(z)
\label{compbeamequ}
\\&
\qquad\qquad
+\tfrac{\ii }{\sqrt{2}}\langle D\rangle^{-1}
f_{1}\Big( {\langle D\rangle^{-1}}{\sqrt{2}}\Re(z),
{\langle D\rangle^{-1/2}}{\sqrt{2}}\Re(w)
\Big)
+\frac{\ii}{\sqrt{2}}\gamma\langle D\rangle^{-1} \mathtt{f}_{b}(t)
\,,\label{compbeamequ2}
\\
\pa_{t}w&= \langle D\rangle\Im(w)
+{\ii}\langle D\rangle^{-1/2}\mathcal{W}\langle D\rangle^{-1/2}\Re(w)
+\ii \beta{\rm Im}(w)
\label{compwaveequ}
\\&
\qquad
+\tfrac{\ii }{\sqrt{2}}\langle D\rangle^{-1/2}f_{2}\Big( {\langle D\rangle^{-1}}{\sqrt{2}}\Re(z),
{\langle D\rangle^{-1/2}}{\sqrt{2}}\Re(w)
\Big)+\frac{\ii}{\sqrt{2}}\delta\langle D\rangle^{-\frac{1}{2}}\mathtt{f}_{w}(t)
\,,\label{compwaveequ2}
\end{align}
where, to shorten the notation, we denoted the nonlinear terms by
\begin{equation*}
f_i(y, \theta)=F_i(y, y_x, y_{xx}, \theta, \theta_x, \theta_{xx}) \qquad i=1, 2
\end{equation*}
and we omitted the equations for the complex conjugates.

\noindent
Consider first the nonlinear term $f_1$ and recall that, by assumption,  $s-2\geq {s}_0>1/2$.
Since $f_1$ is a polynomial of degree two in the unknown 
of the equation we write 
$f_1(y,\theta)=a_1(x)\theta_{xx}^2+\theta_{xx}\check{f}_1^1+\check{f}_1^2$, 
where $\check{f}_1^i$, $i=1,2$, are polynomial independent of $\theta_{xx}$.
By using \eqref{tame}  we have
\begin{equation}\label{restaccio}
\begin{aligned}
\| \langle D\rangle^{-1}\check{f}^2_1\big(\langle D\rangle^{-1}\tfrac{z+\bar{z}}{\sqrt{2}}, 
\langle& D\rangle^{-\frac12}\tfrac{w+\bar{w}}{\sqrt{2}}\big)\|_{H^s}
\lesssim_s 
\|\check{f}^2_1\big(\langle D\rangle^{-1}\tfrac{z+\bar{z}}{\sqrt{2}}, 
\langle D\rangle^{-\frac12}\tfrac{w+\bar{w}}{\sqrt{2}}\big)\|_{H^{s-1}}\,,
\\
&\lesssim_s (\|w\|_{H^{\mathfrak{s}_0+1/2}}+\|z\|_{H^{\mathfrak{s}_0+1}})(\|z\|_{H^s}+\|w\|_{H^{s-1/2}})
\end{aligned}
\end{equation}
which shows that  it
is a contribution to $\mathcal{R}$ as in \eqref{stimaRRR}.
Consider now $a_1(x)\theta_{xx}^2$. By applying first Lemma \ref{ParaMax} 
and the composition Theorem \ref{compo} we obtain
\begin{equation*}
\begin{aligned}
a_1(x)\theta_{xx}^2&=\opbw(a_1(x))\theta_{xx}^2+\opbw(\theta_{xx}^2)\,{a_1}(x)+R^p(\theta_{xx}^2, \,{a_1}(x))
\\&
=\opbw(2 a_1(x)\theta_{xx})\theta_{xx}+\underbrace{R^c_{\frac12}(a_1(x), {2} \theta_{xx})\theta_{xx}}_{(I)}
\\&
+\underbrace{\opbw(a_1(x))R^p(\theta_{xx},\theta_{xx})}_{(II)}
+\underbrace{\opbw(\theta_{xx}^2)\,{a_1}(x)}_{(III)}
+\underbrace{R^p(\theta_{xx}^2,a_1(x))}_{(IV)}\,.
\end{aligned}
\end{equation*}
We now show that the lasts  four summands in the r.h.s of the above equation may be absorbed in the remainder $\mathcal{R}$ after passing in complex coordinates. Consider $(I)$, 
by means of \eqref{resto} we have 
\begin{align*}
\|\langle D\rangle^{-1}R^c_{\frac12}(a_1(x),&{\sqrt{2}}\langle D\rangle^{-\frac12}\Re{w_{xx}})\langle D\rangle^{-\frac12}{\sqrt{2}}\Re{w_{xx}})\|_{H^s}
\\&
\lesssim_s\|\langle D\rangle^{-\frac12}{\sqrt{2}}\Re{w_{xx}}\|_{H^{s-1-1/2}}\|a_1\|_{H^{\mathfrak{s}_0+1/2}}\|\langle D\rangle^{-1/2}{\sqrt{2}}\Re{w_{xx}}\|_{H^{{s}_0+1/2}}
\\
&\lesssim_s \|w\|_{H^s}\|w\|_{H^{{s}_0+2}}\,.
\end{align*}
Hence $(I)$ is a term that can be absorbed in $\mathcal{R}$. 

\noindent
 We study $(II)$, by using \eqref{ac1} and the fact that $a$ is a regular function,  the estimate
 \eqref{eq: paraproductMax2} with $r={s}_0+1/2$ we get
\begin{equation*}
\|\langle D\rangle^{-1}\opbw({a_1}(x))R^p
(\langle D\rangle^{-\frac12} \,\sqrt{2}\Re(w_{xx}),\langle D\rangle^{-\frac12}\,\sqrt{2}\Re(w_{xx}))\|_{H^s}
\lesssim_s \|w\|_{H^s}\|w\|_{H^{{s}_0+2}}\,,
\end{equation*}
which again may be absorbed in $\mathcal{R}$. \\The term $(IV)$ gives the contribution $\|\langle D\rangle^{-1}R^p\big(\langle D\rangle^{-1/2}\Re(w_{xx}/\sqrt{2})^2, a_1(x)\big)\|_{H^s}$. This may be estimated as done for $(II)$, by using also \eqref{tame}. Concerning $(III)$, we have 
\begin{equation*}
\|\opbw\Big(\langle D\rangle^{-\frac12}2\Re(w_{xx})^2\Big)a_1(x)\|_{H^s}\lesssim_s \|w\|_{H^{{s}_0+3/2}}\,.
\end{equation*}
We eventually obtained that the only term which is not a remainder 
coming from the paralinearization of $a_1(x)\theta_{xx}^{2}$ is
\begin{align}
\langle D\rangle^{-1}\opbw\Big(2 a_1(x)&\langle D\rangle^{-\frac12}\,\sqrt{2}\Re{(w_{xx})}\Big)\langle D\rangle^{-\frac12}\, \sqrt{2}\Re(w_{xx})
\nonumber
\\&
=-\opbw\Big(2 a_1(x)\Big(\langle D\rangle^{-\frac12} \,\sqrt{2}\Re(w_{xx})\Big)
\langle \xi\rangle^{-\frac32}\xi^2\Big)\,\sqrt{2}\Re(w)\label{figaro1}
\\&
+R^c_{\frac12}\Big(2\langle\xi\rangle^{-1}a_1(x)
\langle D\rangle^{-\frac12}\,\sqrt{2}\Re(w_{xx}),\xi^2\langle\xi\rangle^{-\frac12}\Big)\,\sqrt{2}\Re(w)\,,
\label{figaro2}
\end{align}
where we have used Theorem \ref{compo} with $\rho=1/2$. 
By \eqref{resto} 
one can show that the term in \eqref{figaro2} contributes to $\mathcal{R}$.
The summand in \eqref{figaro1} is a contribution to the paradifferential operator
with symbol $B_{b}(x,\x)$ (see \eqref{matricitotali}, \eqref{simbolobeamoff}).
Indeed one has (recall \eqref{seminorme})
\begin{equation*}
|a_1(x)\Big(\langle D\rangle^{-\frac12}\,\sqrt{2}\Re(w_{xx})\Big)
\langle \xi\rangle^{-\frac32}\xi^2|_{\frac12,p,\alpha}\lesssim_{p, \alpha} \|w\|_{p+\frac32}\,,
\quad \forall s_0\leq p\leq s-3/2\,, \quad \,\forall \alpha\in \mathbb{N},
\end{equation*}
which is the bound \eqref{figaro3}.
Reasoning analogously (applying Theorem \ref{compo} and  Lemma \ref{ParaMax}) one proves that the term 
$\theta_{xx}\check{f}_1^1$, passing to complex coordinates reads
\begin{equation*}
\opbw(\check{f}_1^1(\langle D\rangle^{-1}\sqrt{2}\Re(z),\langle D\rangle^{-\frac12}\sqrt{2}\Re(w))\langle\xi\rangle^{\frac12})\sqrt{2}\Re(w)+\widetilde{\mathcal{R}}(z,w)\,,
\end{equation*}
for some $\widetilde{\mathcal{R}}$ satisfying \eqref{stimaRRR}.
Hence we proved that 
\begin{equation*}
\begin{aligned}
\langle D\rangle^{-1}
&f_{1}\Big( \langle D\rangle^{-1}\sqrt{2}\Re(z),
\langle D\rangle^{-1/2}\sqrt{2}\Re(w)
\Big)= \\
&-\opbw\Big(\partial_{\theta_{xx}}f_1(\langle D\rangle^{-1}\sqrt{2}\Re(z),\langle D\rangle^{-\frac12}\sqrt{2}\Re(w)\Big)\langle\xi\rangle^{-\frac32}\xi^2\Big)\sqrt{2}\Re(w)
\end{aligned}\end{equation*}
up to terms contributing to the remainder $\mathcal{R}$ as in \eqref{stimaRRR}.

Let us now consider the nonlinear term $f_2$. This has the form
\begin{equation}\label{f2lifrati}
f_2(y,\theta)=a_2(x)y_{xx}^2+y_{xx}\check{f}_2^1(\check{y}_{xx})
+a_3(x)\theta_{xx}^2+\theta_{xx}\check{f}_2^2(\check{\theta}_{xx})
+\check{f}_2^3(\check{\theta}_{xx},\check{y}_{xx})\,,
\end{equation}
where $\check{\theta}_{xx}$ and $\check{y}_{xx}$ 
means that the function depends on all the variables but these ones. 
We proceed as for $f_1$. By using Lemma \ref{ParaMax} and Theorem \ref{compo} 
we infer
\begin{equation*}
\begin{aligned}
a_2(x)y_{xx}^2&=\opbw(a_2(x))y_{xx}^2+\opbw(y_{xx}^2)a_2(x)+R^{p}(y_{xx}^2,a_2(x))
\\&
=\opbw(a_2(x))(2\opbw(y_{xx})y_{xx}+R^p(y_{xx},y_{xx}))+R^{p}(y_{xx}^2,a_2(x))
\\&
=-2\underbrace{\opbw(a_2(x)y_{xx}\xi^2)y}_{(I)}-2 \underbrace{R^c_{\frac12}(a_2(x)y_{xx},\xi^2)y}_{(II)} 
+ \underbrace{\opbw(a_2(x))R^p(y_{xx},y_{xx})}_{(III)}
\\&
\,\,\,\,+\underbrace{R^{p}(y_{xx}^2,a_2(x))}_{(IV)}\,.
\end{aligned}
\end{equation*}
After passing in complex coordinates the terms $(II)$, $(III)$ and $(IV)$ 
give a contribution to $\mathcal{R}$. We prove this fact for $(II)$. One has
\begin{equation*}
\begin{aligned}
\|\langle D\rangle^{-\frac12}R^c_{\frac12}(a_2(x)&\langle D\rangle^{-1}\Re(z_{xx}),\xi^2)
\langle D\rangle^{-1} \Re(z)\|_{H^s}
\\&
\leq\|R^c_{\frac12}(a_2(x)\langle D\rangle^{-1}\Re(z_{xx}),\xi^2)
\langle D\rangle^{-1} \Re(z)\|_{H^{s-1/2}}
\\&
\lesssim_s \|\langle D\rangle^{-1}z_{xx}\|_{H^{s_0+1/2}}\|\langle D\rangle^{-1}z\|_{H^{s+2-1/2-1/2}}
\\&
\lesssim_s \|z\|_{H^{s_0+3/2}}\|z\|_{H^{s}}\,,
\end{aligned}
\end{equation*}
where we have used \eqref{resto}. 
The proofs for $(III)$ and $(IV)$ are similar to the ones performed previously for $f_1$, 
therefore they are omitted. 
The addend $(I)$ is not a remainder and gives a contribution 
to the term of order $1/2$. 
Indeed by Theorem \ref{compo} we have
\begin{equation*}
\begin{aligned}
\langle D\rangle^{-\frac12}
\opbw(a_2(x)\langle D\rangle^{-1}\Re(z_{xx})\xi^2)&\langle D\rangle^{-1}\Re(z)
\\&=\opbw(a_2(x)\langle D\rangle^{-1}\Re(z_{xx})\xi^2\langle\xi\rangle^{-\frac32})\Re(z)\,,
\end{aligned}
\end{equation*}
modulo remainders that may be absorbed in $\mathcal{R}$.
The symbol of the paradifferential operator above is a contribution to 
the symbol $B_{w}(x,\x)$ (see \eqref{matricitotali}, \eqref{simbolowaveoff})
and satisfies \eqref{figaro3}.

\smallskip

Coming back to \eqref{f2lifrati}, we restrict the discussion to the paralinearization of the term $a_3(x)\theta_{xx}^2$,  the proof for $\theta_{xx}\check{f}_2^2(\check{\theta}_{xx})$ being very similar.  Finally, the term $\check{f}_2^3$ has to be handled exactly as $\check{f}_1^2$ in \eqref{restaccio}. We have
\begin{align*}
a_3(x)\theta_{xx}^2&=-2{\opbw(a_3(x)\theta_{xx}\xi^2)\theta}-2{R^c_{1}(a_3(x)\theta_{xx},\xi^2)\theta} + {\opbw(a_3(x))R^p(\theta_{xx},\theta_{xx})}\\
&\,\,\,\,+{R^{p}(\theta_{xx}^2,a_3(x))}.
\end{align*}
We study the contribution coming from the first two summands in the r.h.s, all the other summand may be absorbed in $\mathcal{R}$ reasoning as above. The contribution, in complex coordinates, of the first summand is
\begin{equation}\label{martinez}
\begin{aligned}
-4\langle D\rangle^{-1/2}\opbw(a_3(x)\langle D\rangle^{-\frac12}
&\Re(w_{xx})\xi^2)\langle D\rangle^{-1/2}\Re(w)
\\&
=-4\opbw(a_{3}(x)\langle D\rangle^{-\frac12}\Re(w_{xx})\xi^2 \langle{\xi}\rangle^{-1})\Re(w)\,,
\end{aligned}\end{equation}
modulo a remainder that may be absorbed in  $\mathcal{R}$.
In particular we note that (recall Def. \ref{def:simbolini})
\begin{equation}\label{diff-simbo}
\frac{\x^{2}}{\langle\x\rangle}-|\x|=-\frac{|\xi|}{\xi+\langle\xi\rangle}
\frac{1}{\langle\xi\rangle}\in \Gamma^{-1}_{s}\,,\quad \forall s\in\R\,.
\end{equation}
Therefore
\begin{equation}\label{figaro5}
\eqref{martinez}=
-4\opbw(a_{3}(x)\langle D\rangle^{-\frac12}\Re(w_{xx})|\xi|){\rm Re}(w)
+\widetilde{Q}_{2}(z,w)\,,
\end{equation}
where
\[
\widetilde{Q}_{2}(z,w)=
\opbw\Big(a_{3}(x)\langle D\rangle^{-\frac12}\Re(w_{xx})
\big(\frac{\x^{2}}{\langle\x\rangle}-|\x| \big) \Big){\rm Re}(w)\,.
\]
The symbol in the paradifferential operator in \eqref{figaro5}
can be absorbed in the symbol \eqref{matricitotali} 
(see also \eqref{matriceregowave}-\eqref{simbolowave})
and satisfies  \eqref{figaro4}.
Moreover, 
using Theorem \ref{azione}, we deduce that $\widetilde{Q}_2$ 
can be absorbed in the remainder 
$\mathcal{R}$ satisfying \eqref{stimaRRR}. We explain why $-2{R^c_{1}(a_3(x)\theta_{xx},\xi^2)\theta}$ may be absorbed in the remainder $\mathcal{R}$: passing in complex coordinates we have 
\begin{align*}
\|\langle D\rangle^{-\frac12}R_1^c
(\langle D \rangle^{-\frac12}&\Re(w_{xx}),\xi^2)\langle D\rangle^{-\frac12}\Re(w)\|_{H^s}
\\&
\leq 
\| R_1^c(\langle D \rangle^{-\frac12}\Re(w_{xx}),\xi^2)\langle D\rangle^{-\frac12}\Re(w)\|_{H^{s-1/2}}
\\&
\lesssim_s 
\|\langle D\rangle^{-\frac12}w_{xx}\|_{s_0+1}\|\langle D\rangle^{-\frac12}w\|_{H^{s+2-1-1/2}}
\\&
\lesssim_s \|w\|_{H^{s_0+5/2}}\|w\|_{s}\,.
\end{align*}
This concludes the discussion on the nonlinear part.

\smallskip
We now rewrite the linear part of system \eqref{sistemaponte} 
as paradifferential operators modulo bounded remainders.
By
 \eqref{opastratti}, \eqref{opside1} in Lemma \ref{lem:formexplicBW}, 
 and using  Theorem \ref{compo} 
we infer that
\[
\langle D\rangle^{-1}\mathcal{B}\langle D\rangle^{-1}
=\opbw\big(- b(x)|\x|^2 -2\ii b_x(x)\x\big)\,+\widetilde{\mathtt{R}}_1\,,
\]
for some $\widetilde{\mathtt{R}}_1$
 satisfying \eqref{restilineari}.
Similarly, using \eqref{opastratti}, \eqref{opside1}  
and Theorem \ref{compo}, we get 
\[
\langle D\rangle^{-1/2}\mathcal{W}\langle D\rangle^{-1/2}
=\opbw(-c(x)|\x|)+\widetilde{\mathtt{R}}_2\,,
\]
for some $\widetilde{\mathtt{R}}_2$
 satisfying \eqref{restilineari}.
 It is enough (recalling \eqref{compbeamequ}, \eqref{compwaveequ})
 to define
 \[
 \mathtt{R}_1:=\widetilde{\mathtt{R}}_1+\frac{\alpha}{2}\sm{1}{-1}{-1}{1}\,,
 \qquad 
  \mathtt{R}_2:=\widetilde{\mathtt{R}}_2+\frac{\beta}{2}\sm{1}{-1}{-1}{1}
 \]
 to get item $(ii)$.
\end{proof}

\begin{remark}\label{rmk:datiiniziali}
Proposition \ref{paraparaProp} implies that
given solutions (if they exist) $V(t)=(Z(t),W(t))$ of \eqref{ponteparato}
(see also \eqref{inverse})
with initial conditions 
\begin{equation}\label{tavolo}
Z(0)=Z_0=\vect{z_0}{\bar{z}_0}\in \mathcal{H}^{s}\,,\qquad
W(0)=W_0=\vect{w_0}{\bar{w}_0}\in \mathcal{H}^{s}\,,
\end{equation}
one can recover solutions $(y(t), \theta(t))$ 
of the system \eqref{sistemaponte} in real coordinates
by  means of formula \eqref{inverse} and 
with initial conditions 
\[
\begin{aligned}
y_0:=y(0)&=\frac{1}{\sqrt{2}}\langle D\rangle^{-1}(z_0+\bar{z}_0)\,,\qquad 
y_1:=(\pa_{t}y)(0)=\frac{1}{\ii\sqrt{2}}\langle D\rangle(z_0-\bar{z}_0)\,,
\\
\theta_0:=\theta(0)&=\frac{1}{\sqrt{2}}\langle D\rangle^{-1/2}(w_0+\bar{w}_0)\,,\qquad 
\theta_1:=(\pa_{t}\theta)(0)=\frac{1}{\ii\sqrt{2}}\langle D\rangle^{1/2}(w_0-\bar{w}_0)\,.
\end{aligned}
\]
By using \eqref{tavolo} we have
\[
(y_0,y_1)\in H^{s+1}\times H^{s-1}\,,
\qquad
(\theta_0,\theta_1)\in H^{s+\frac{1}{2}}\times H^{s-\frac{1}{2}}\,.
\]
\end{remark}

\section{Modified energy}\label{sec:modene}
We consider a simplified version of equation \eqref{ponteparato}.
More precisely, recalling \eqref{parametrifissati}, we  consider a fixed function  
$\widetilde{V}=(\widetilde{Z}, {\widetilde{W}} )\in L^{\infty}([0,T];\mathcal{H}^{s_*}\times\mathcal{H}^{s_*})$ with $s_*\geq s_1+2$,
for some $T>0$,
such that 
\begin{align}
&\|\widetilde{Z}\|_{L^{\infty}H^{s_1}}\leq r\,, 
\quad 
\|\widetilde{Z}\|_{L^{\infty}H^{s_1+2}}
+\|\partial_t\widetilde{Z}\|_{L^{\infty}H^{s_1}}\leq R\,,\label{palla-beam}
\\&
\|\widetilde{W}\|_{L^{\infty}H^{s_1}}\leq r\,, 
\quad 
\|\widetilde{W}\|_{L^{\infty}H^{s_1+1}}
+\|\partial_t\widetilde{W}\|_{L^{\infty}H^{s_1}}\leq R\,,\label{palla-wave}
\end{align}
{for some $r, R>0$.}
%
{Let $s\geq 0$, } we consider the linear problem
\begin{equation}\label{linprob}
\pa_{t}V=\mathfrak{A}(\widetilde{V})V+
\mathfrak{B}(\widetilde{V})V\,,\qquad V(0)=V_0=\vect{Z_0}{W_0}\in \mathcal{H}^{s}\times \mathcal{H}^{s}\,,
\end{equation}
where the operators 
$\mathfrak{A}(\widetilde{V})$, $\mathfrak{B}(\widetilde{V})$ are the ones given 
in Proposition \ref{paraparaProp}. The aim of this section is to provide \emph{a priori} estimates on the solutions of \eqref{linprob}. To this end, we shall diagonalize, {modulo bounded remainders,} the operator $\mathfrak{A}(\widetilde{V})+\mathfrak{B}(\widetilde{V})$. In order to obtain so, we shall proceed as follows.\\
- We first construct a \emph{parametrix} $\Psi(\widetilde{V})$, i.e. an approximatively invertible map on $\mathcal{H}^s{\times\mathcal{H}^s}$, which conjugates  $\mathfrak{A}(\widetilde{V})+\mathfrak{B}(\widetilde{V})$ to a diagonal operator, modulo semilinear remainders. This is the content of Proposition \ref{prop:costruzioneMappa}. \\
- Then, by means of this parametrix, we construct a \textit{modified norm} on $\mathcal{H}^s\times\mathcal{H}^s$ which is tailored to the problem \eqref{linprob} and, in Lemma \ref{equivalenzatotale}, we  prove that this norm is {almost} equivalent to the classical norm on $\mathcal{H}^s\times\mathcal{H}^s$. This enables us to use this modified norm to prove the desired energy bound, this is the content of Proposition \ref{stima-energia-totale}. To  achieve the energy estimates, a fundamental ingredient is a Garding-type inequality, this is the content of Lemma \ref{Garding-deltone}.

\medskip

\textbf{Notation.} For convenience, in the rest of the paper, we use the notation $C_r$ and $C_R$ to denote constants that depend on $r$ and $R$ respectively, and possibly on the Sobolev-regularity indexes and the Sobolev norm of the functions $a, b, c, d$. These constants may change from line to line. We always suppose that $C_R\gg C_r$, therefore we shall write $C_rC_R=C_R$.

\subsection{Construction of the parametrix}
In the following proposition we construct the aforementioned parametrix.
\begin{proposition}\label{prop:costruzioneMappa}
{Let $s\geq 0$}.
Assume \eqref{palla-beam}-\eqref{palla-wave}, then there are maps $\Phi=\Phi(\widetilde{V})$, $\Psi=\Psi(\widetilde{V})$
in  $\mathcal{L}\big(\mathcal{H}^{s}\times\mathcal{H}^{s};\mathcal{H}^{s}\times\mathcal{H}^{s}\big)
$
satisfying the following:

\noindent
$(i)$ The equality
\begin{equation}\label{fessemamt}
\Phi\circ (\mathfrak{A}(\widetilde{V})+\mathfrak{B}(\widetilde{V}))\circ \Psi=\Lambda+\mathcal{M} 
\end{equation}
holds, where $\mathcal{M}\in 
\mathcal{L}\big(\mathcal{H}^{s}\times\mathcal{H}^{s};\mathcal{H}^{s}\times\mathcal{H}^{s}\big)$
and 
\begin{equation}\label{diagonalissima}
\Lambda=\Lambda(\widetilde{V})=\left( \begin{matrix}-\ii E\opbw(\lambda_{b}\,\x\,^{2}) & 0 \\0
& -\ii E\opbw(\lambda_{w}|\x|)\end{matrix}\right)\,,
\end{equation}
for some functions $\lambda_b(x)$, $\lambda_w(x)=\lambda_{w}(\widetilde{V};x)$ satisfying 
\begin{equation}\label{olio2}
\begin{aligned}
\|\lambda_b-1\|_{H^{p}}&\lesssim_{p} 1\,,\qquad \qquad \quad\,\,\quad {p>1/2}
\\
\|\lambda_w-1\|_{H^{p}}&\lesssim_{p}\|\widetilde{W}\|_{H^{p+3/2}}\,,\qquad {s_0\le p\leq s_*-3/2\,.}
\end{aligned}
\end{equation}
Moreover 
\begin{equation}\label{olio5}
\|\mathcal{M}V\|_{H^s}\lesssim_{s} \|V\|_{H^{s}}\|\widetilde{V}\|_{H^{s_1}}\,,\quad 
\forall\, V\in \mathcal{H}^{s}\times\mathcal{H}^{s}\,.
\end{equation}

\noindent
$(ii)$ The operator  $\Psi\circ \Phi-\Id_{\mathbb{C}^4}$
belongs to
 $\mathcal{L}\big(\mathcal{H}^{s}\times\mathcal{H}^{s};\mathcal{H}^{s+2}\times\mathcal{H}^{s+2}\big)$, and we have the estimates
\begin{equation}\label{stimephipsi}
\begin{aligned}
\|\Phi(\widetilde{V})[\cdot]\|_{\mathcal{L}(H^s;H^s)}\,, 
\, \|\Psi(\widetilde{V})[\cdot]\|_{\mathcal{L}(H^s;H^s)}&\leq C_r\,,
\\
\|\Psi\circ \Phi-\Id_{\mathbb{C}^4}\|_{{\mathcal{L}(H^s;H^{s+2})}}&\leq {C_{R}}\,.
\end{aligned}
\end{equation} 

 \noindent
 $(iii)$ The map $\Phi$ has the form
 \begin{equation}\label{expaMappa}
 \Phi=\left(\begin{matrix}\mathtt{D}_{b} & 0 \\ 0 & \mathtt{D}_{w}
 \end{matrix}\right)\circ \left( \left(\begin{matrix}\Id_{\mathbb{C}^2} & 0 \\ 0 &\Id_{\mathbb{C}^2}\end{matrix}\right)  + 
 \left(\begin{matrix}0 & \mathtt{T}_1 \\ \mathtt{T}_2 &0\end{matrix}\right) 
 \right),
 \end{equation}
 where $\mathtt{T}_i=\opbw(T_{i}(x,\x))$ for some symbols 
 satisfying 
 \begin{equation}\label{olio3}
 |T_{i}|_{-\frac{3}{2},p,\alpha}\lesssim_{p, \alpha} \| \widetilde{W} \|_{H^{p+3/2}}
\,,\quad s_0\le p\leq s_*-3/2,\,\,\forall \alpha\in\mathbb{N}\,.  
 \end{equation}
 Moreover the operators $\mathtt{D}_{b},\mathtt{D}_{w}$
 satisfy
 \begin{equation}\label{olio}
 \|\mathtt{D}_{b}\|_{\mathcal{L}(H^s;H^s)}\,, \|\mathtt{D}_{w}\|_{\mathcal{L}(H^s;H^s)}\leq C_{r} \qquad {\forall s\geq 0}\,.
\end{equation}

\noindent $(iii)$ The operator $\pa_{t}\Phi$
 belongs to 
 $\mathcal{L}\big(\mathcal{H}^{s}\times\mathcal{H}^{s};\mathcal{H}^{s}\times\mathcal{H}^{s}\big)$
 and 
 \begin{equation}\label{stimatempoaltempo}
 \|\pa_{t}\Phi\|_{\mathcal{L}(H^{s};H^{s})}\leq C_{R}\,.
 \end{equation}

 
\end{proposition}
The proof of this proposition is divided in several steps. 
Actually, the operators $\mathtt{D}_{b},\mathtt{D}_{w}$ are paradifferential and we will provide an explicit expression for them
in the following.
First we need to diagonalize, separately,  the principal order operators in \eqref{matriceregobeam} and \eqref{matriceregowave}, this is done  in Lemmata \ref{lem:diagobeam} and \ref{lem:diagowave}. After that, in Lemma \ref{lem:diago-beam-uno}, we  diagonalize the first order matrix in the \eqref{matriceregobeam}, while in Lemma \ref{eliminazione} we eliminate the diagonal term of order one in \eqref{matriceregobeam}. All these changes of coordinates allow us to diagonalize the system \eqref{linprob} at the principal and subprincipal orders. 

We eventually eliminate, in the proof of the Proposition \ref{prop:costruzioneMappa}, the off diagonal term {of} $\mathfrak{B}(\widetilde{V})$ of order $1/2$.

\subsubsection{Diagonalization of the principal order for the beam.}
We define the following quantities related to the highest order matrix of the beam equation in  \eqref{matriceregobeam}
\begin{equation}\label{quantit-beam}
\begin{aligned}
&\lambda_b(x):=\sqrt{1+2a(x)}\,, 
\quad 
S_b:=\left(\begin{matrix}s_{{1,b}} &s_{2,b}\\
s_{2,b} &s_{1,b}\end{matrix}\right)\,, 
\quad 
S_b^{-1}=\left(\begin{matrix}s_{{1,b}} &-s_{2,b}\\
-s_{2,b} &s_{1,b}\end{matrix}\right)
\\&
s_{1,b}:=\frac{1+a(x)+\sqrt{1+2a(x)}}{\sqrt{2\lambda_b(x)(1+a(x)+\lambda_b(x))}}\,,
\quad s_{2,b}:=\frac{-a(x)}{\sqrt{2\lambda_b(x)(1+a(x)+\lambda_b(x))}}\,.
\end{aligned}
\end{equation}
{We point out that $\lambda_b(x)=\sqrt{b(x)}$ is well defined thanks to the hypothesis \ref{hyp1} ( the ellipticity condition).}

In the above quantities, the functions $\pm\lambda_b(x)$ 
are the eigenvalues of the matrix 
$E(\Id_{\C^2}+\mathbf{U}a(x))$, while {$S_{b}$} 
is the matrix of the eigenvectors associated to $E(\Id_{\C^2}+\mathbf{U}a(x))$. 
For these reasons we have 
\begin{equation}\label{diagobeam}
S_b^{-1}E(\Id_{\C^2}+\mathbf{U}a(x))S_b= E \lambda_b(x)\,.
\end{equation}
We have the following.

 \begin{lemma}\label{lem:diagobeam}
Recall \eqref{parametrifissati},  \eqref{matriceregobeam}. 
 One has that
 \begin{equation}\label{cliff}
 \|s_{1,b}-1\|_{H^{p}}+\|s_{2,b}\|_{H^{p}}+\|\lambda_{b}-1\|_{H^{p}}
 \lesssim_{p} 1 \qquad \forall p\geq s_0\,.
 \end{equation}
 Moreover
 \begin{align}
\!\!\! \opbw(S_b^{-1})E\opbw\big(A_b\big)\opbw(S_b)&= 
 E \opbw\Big(\lambda_b(x)\xi^2\Id_{\C^2}+2\ii{\bf U} a_x(x)\xi\Big)
 +R_{0,b}\,,\label{diago2beam}
 \end{align}
 where $R_{0,b}$ belongs to $\mathcal{L}(\mathcal{H}^s; \mathcal{H}^s)$ 
 for $s\geq 0$, {with operator norm bounded form above by some $C>0$
 depending on $s$ and $\|a\|_{H^{{s}_0+1}}$ }.
 \end{lemma}
 \begin{proof}
 The bound \eqref{cliff} follows by the explicit expression \eqref{quantit-beam}
 and classical composition estimates on Sobolev spaces.
Formula \eqref{diago2beam} 
follows by \eqref{diagobeam} and Theorem \ref{compo} with $\rho=2$ .
 \end{proof}
 \subsubsection{Diagonalization at subprincipal order for the beam equation.}
 Here we aim at diagonalizing the subprincipal matrix of symbols in the r.h.s. of \eqref{diago2beam}. We define the following symbol and the related operator of order $-1$
 \begin{equation}\label{generatore-diago-ord-1}
\mathtt{M}_{-1}:=\opbw\left(\begin{matrix}
 0&m(x,\xi)\\
 m(x,\xi)&0
\end{matrix}\right)\,, \quad m(x,\xi):=-\ii \frac{a_x(x)}{\lambda_b(x)\xi}\psi(\x)\,, 
 \end{equation}
 for some $\psi\in C^{\infty}(\R)$ satisfying $\psi(\x)=1$ if $|\x|\geq 1/2$ 
 and $\psi(\x)=0$ for $|\x|\leq 1/4$. {We remark that $\lambda_b(x)$ never vanishes thanks to the ellipticity condition \ref{hyp1}.}

 We have the following result.
 \begin{lemma}\label{lem:diago-beam-uno}
Recall \eqref{parametrifissati}.
 We have
 that $m=m(x,\x)$  belongs to $\Gamma^{-1}_{p}$, $p\geq {s}_0$ and
 satisfies 
 \begin{equation}\label{cliff2}
 |m|_{-1,p,\alpha}\lesssim_{p,\alpha} 1\,,\qquad \forall\alpha\in \N\,.
 \end{equation} 
 Moreover, recalling \eqref{diago2beam}, 
 \begin{equation}\label{diagodiago}
 \begin{aligned}
 (\Id_{\C^2}+\mathtt{M}_{-1})E
 \opbw\Big(\lambda_b(x)\xi^2\Id_{\C^2}+2\ii{\bf U} a_x(x)\x\Big)
 &(\Id_{\C^2}-\mathtt{M}_{-1})
 \\&=
 E\opbw\Big(\lambda_b(x)\xi^2+2\ii a_x(x)\xi\Big)+R_0,
\end{aligned}
 \end{equation}
 where $R_0$ belongs to ${\mathcal{L}(\mathcal{H}^s; \mathcal{H}^s)}$
 for $s\geq 0$,
  {with operator norm bounded form above by some $C>0$
 depending on $s$ and $\|a\|_{H^{{s}_0+2}}$.}
 \end{lemma}
 
 \begin{proof}
 Recalling that the function $a(x)$ is $C^{\infty}$, $\lambda_b(x)\geq \sqrt{\mathtt{c}_1}$ for all $x$ (see hypothesis \ref{hyp1}) and 
 that ${\rm supp}(\pa_{x}^{k}\psi)\subset\{1/4\leq |\x|\leq 1/2\}$ for any $k\geq 1$,
 it follows easily that $m\in \Gamma^{-1}_p$, $p\geq s_0$,
 with bound as in \eqref{cliff2}.
 Moreover, by the composition Theorem \ref{compo}, used with $\rho=2$, we get
  formula 
 \eqref{diagodiago}  up to a term of the form
 \[
 \opbw\Big(\sm{0}{1}{1}{0}r(x,\x)\Big)\,,\qquad r(x,\x):=2
 \lambda_b(x) \,\xi^{2}m(x,\x)+2\ii a_{x}(x)\x\,.
 \]
 Using \eqref{generatore-diago-ord-1},  we note that
 \[
 r(x,\x)=2\ii a_{x}(x)\x (1-\psi(\x))\,,
 \]
 which is compactly supported in $\xi$ and therefore can be absorbed in the operator $R_0$.
 \end{proof}

\subsubsection{Cancellation of the self-adjoint term in the beam equation.} \label{magia}
In the following we cancel the self-adjoint part of the 
subprincipal symbol for the beam equation. We perform a Gauge-type transformation in the spirit of \cite{Ozawa}.
We consider the function 
\begin{equation}\label{generatore-eliminazione}
k(x)=\exp(2\lambda_b(x)).
\end{equation}
We have the following lemma.
\begin{lemma}\label{eliminazione}
We have that $k\in \Gamma^{0}_s$, $s\geq{s}_0>1/2$,  and 
  satisfies 
 \begin{equation}\label{cliff3}
 |k|_{0,s,\alpha}\lesssim_{s,\alpha}\|a\|_{H^{s+1}}\,,
 \qquad \forall\alpha\in \N\,.
 \end{equation} 
 Moreover
\begin{equation*}
\opbw(k^{-1}(x))\opbw\Big(\lambda_b(x)\,\xi^2+2\ii a_x(x)\xi\Big)\opbw(k(x))
=\lambda_b(x)\xi^2+R_0\,,
\end{equation*}
where $R_0$ belongs to $\mathcal{L}(H^{s};H^{s})$ for any $s\geq0$,
  {with operator norm bounded form above by some $C>0$
 depending on $s$ and $\|a\|_{H^{{s}_0+2}}$.}
\end{lemma}
\begin{proof}
One has to apply Theorem \ref{compo} with $\rho=2$.
\end{proof}
Resuming all the lemmas of this section, we  built a \emph{parametrix} which diagonalizes the principal orders of the beam equation. More precisely
we have the following.
Recalling \eqref{matriceregobeam}, we set
 \begin{align}
 \mathtt{D}_{b}&:=\opbw(k^{-1}(x))(\Id_{\C^2}+\mathtt{M}_{-1})\opbw(S_b^{-1})\,,
 \label{parametrix}
 \\
 \widetilde{\mathtt{D}}_{b}&:=\opbw(S_b)(\Id_{\C^2}-\mathtt{M}_{-1})\opbw(k(x))\,,
 \label{inverse2}
 \\
 \mathcal{R}_{-2}&:=\widetilde{\mathtt{D}}_{b}\mathtt{D}_{b}-\Id_{\C^2}\,,
 \label{approx-inv}
 \end{align}
 and we note that, by using Theorem \ref{azione}, the bounds \eqref{cliff}, \eqref{cliff2}, \eqref{cliff3},
 and Theorem \ref{compo} with $\rho=2$,
 one has 
 \begin{equation}\label{totalMappa}
 \| \mathtt{D}_{b}\|_{\mathcal{L}(H^{s};H^s)}\,,\;
  \| \widetilde{\mathtt{D}}_{b}\|_{\mathcal{L}(H^{s};H^s)}\,,
   \|\mathcal{R}_{-2}\|_{\mathcal{L}(H^{s};H^{s+2})}\lesssim_{s}C(\|a\|_{H^{{s}_0+2}})\,,
   \quad {s\geq0} \,\,\,,
 \end{equation}
 for some constant $C(\|a\|_{H^{{s}_0+2}})$ depending on $s$ and $\|a\|_{H^{{s}_0+2}}$.
Moreover $\mathtt{D}_{b}$ 
 essentially diagonalizes the operator $\opbw(A_b(x,\xi))$ in the  sense that 

\begin{equation}\label{total-diago}
\begin{aligned}
&\mathtt{D}_{b}E\opbw(A_b(x,\xi))\widetilde{\mathtt{D}}_{b}=E\opbw\Big(\lambda_b(x)\,\xi^2\Big)+R_{0,b}\\
\end{aligned}
\end{equation}
for some remainder $R_{0,b}$ belonging to $\mathcal{L}(\mathcal{H}^{s};\mathcal{H}^{s})$,
with operator norm bounded 
from above by a constant $C>0$ depending only on $s$ and $\|a\|_{H^{{s}_0+2}}$.

\subsubsection{Diagonalization of the principal order for the wave}
Analogously to what we have done for the beam equation,  
for the wave equation 
we define
\begin{equation}\label{quantit-wave}
\begin{aligned}
&\lambda_w:=\lambda_w(\widetilde{Z},\widetilde{W};x):=\sqrt{1+2a_w}\,, 
\quad 
S_w:=\left(\begin{matrix}s_{{1,w}} &s_{2,w}\\
s_{2,w} &s_{1,w}\end{matrix}\right)\,, 
\quad 
S_w^{-1}=\left(\begin{matrix}s_{{1,w}} &-s_{2,w}\\
-s_{2,w} &s_{1,w}\end{matrix}\right),
\\&
s_{1,w}:=\frac{1+a_w+\sqrt{1+2a_w}}{\sqrt{2\lambda_w(1+a_w+\lambda_w)}}\,,
 \quad s_{2,b}:=\frac{-a_w}{\sqrt{2\lambda_w(1+a_w+\lambda_w)}}\,.
\end{aligned}
\end{equation}
{We point out that $\lambda_w(x)=\sqrt{c(x)+\partial_{\theta_{xx}} F_2}$ is well defined thanks to hypothesis \ref{hyp1}.}
The above quantities have the same meaning 
of the ones in \eqref{quantit-beam} but for the wave equation, 
for this reason we infer that
 \begin{equation}\label{diagowave}
 S_{w}^{-1}E(\Id_{\C^2}+\mathbf{U}a_w)S_w=E\lambda_w(\widetilde{Z},\widetilde{W};x)\,.
 \end{equation}
 \begin{lemma}\label{lem:diagowave}
Recall \eqref{parametrifissati}, \eqref{matriceregowave}.
  One has that, for $1/2<p\leq {s}_1$,
 \begin{equation}\label{cliffwave}
 \|s_{1,w}-1\|_{H^{p}}+\|s_{2,w}\|_{H^{p}}
 +\|\lambda_{w}-1\|_{H^{p}}\lesssim_{p}C_{r}\,.
 \end{equation}
 Moreover
 \begin{align}
  \opbw(S_w^{-1})E\opbw\big(A_w\big)\opbw(S_w)&= 
  E \opbw(\lambda_w(x)|\xi|)+R_{0,w}\,, \label{diago1wave}
 \end{align}
 where
 \[
 \|R_{0,w}\|_{\mathcal{L}(H^{\s}; H^{\s})}\lesssim C_{r}\,, \quad \s\geq0\,.
 \]
 Finally one has that 
 \begin{equation}\label{cliffwave200}
  \|\pa_{t}s_{1,w}\|_{H^{s_0}}+\|\pa_{t}s_{2,w}\|_{H^{s_0}}
 +\|\pa_{t}\lambda_{w}\|_{H^{s_0}}\lesssim C_{R}\,, 
 \end{equation}
 uniformly in $t\in [0,T]$.
  \end{lemma}
 
 \begin{proof}
 Using \eqref{quantit-wave}-\eqref{simbolowave}-\eqref{simbog} 
 and hypothesis \eqref{palla-beam}-\eqref{palla-wave} 
 to estimate the semi-norms of the symbols $A_w$ 
 and $S_w$ the thesis follows by using Theorem \ref{compo} with $\rho=1$.
 It remains to show \eqref{cliffwave200}.
 Recalling \eqref{quantit-wave}-\eqref{simbolowave}-\eqref{simbog}, 
the definition of semi-norm \eqref{seminorme}  and 
hypothesis \eqref{palla-beam}-\eqref{palla-wave} (see also \eqref{cliffwave}) 
we have for any $t\geq 0$, 
\begin{equation*}
|\partial_t\lambda_w|_{0, s_0, 4}
\lesssim 
\|\partial_t \widetilde{W}\|_{L^{\infty}H^{{s}_1}}+\|\partial_t \widetilde{Z}\|_{L^{\infty}H^{s_1}}
\lesssim R\,.
\end{equation*}
The others are similar.
 \end{proof}

\noindent
Recalling \eqref{quantit-wave}, we set 
\begin{equation}\label{parametrix-wave}
 \mathtt{D}_{w}:=\opbw(S_w^{-1})\,,
 \qquad
   \widetilde{\mathtt{D}}_{w}:=\opbw(S_w)\,,
   \qquad
    \mathcal{R}_{-2}:=\widetilde{\mathtt{D}}_{w}\mathtt{D}_{w}-\Id_{\C^2}\,,
\end{equation}
 and we note that, by using Theorem \ref{azione}, \eqref{cliffwave}, 
 and Theorem \ref{compo} with $\rho=1$,
 one has
 \begin{equation}\label{totalMappaWave}
 \| \mathtt{D}_{w}\|_{\mathcal{L}(H^{s};H^s)}\,,\;
  \| \widetilde{\mathtt{D}}_{w}\|_{\mathcal{L}(H^{s};H^s)}\,,
   \|\mathcal{R}_{-2}\|_{\mathcal{L}(H^{s};H^{s+2})}\lesssim_{s}C_{r}\,,
   \quad s\geq0\,.
 \end{equation}

\subsubsection{Construction of the parametrix.} This subsection 
is devoted to the proof of Proposition \ref{prop:costruzioneMappa}.
\begin{proof}[{\bf Proof of Proposition \ref{prop:costruzioneMappa}.}] 
We set the notation
\begin{equation}\label{matricitt}
\mathtt{T}:= \left(\begin{matrix}0 & \mathtt{T}_{1} \\ \mathtt{T}_2 &0\end{matrix}\right)\,,
\qquad 
\mathtt{D}:=\left(\begin{matrix}\mathtt{D}_{b} & 0 \\ 0 & \mathtt{D}_{w}
 \end{matrix}\right)\,,\qquad 
 \widetilde{\mathtt{D}}:=\left(\begin{matrix}\widetilde{\mathtt{D}}_{b} & 0 \\ 0 & \widetilde{\mathtt{D}}_{w}
 \end{matrix}\right)\,,
\end{equation}
where $\mathtt{D}_b, \widetilde{\mathtt{D}}_b$ are  defined in \eqref{parametrix}, \eqref{inverse2}, 
$\mathtt{D}_w, $ $\widetilde{\mathtt{D}}_w$ are defined in \eqref{parametrix-wave}
and 
$\mathtt{T}_i=\opbw(T_{i}(x,\x))$, $i=1,2$ are paradifferential operators
whose symbols have to be determined.
We define
\begin{equation*}
{\Phi}:= (\Id_{\mathbb{C}^4}+\mathtt{T})\mathtt{D}\,, 
\quad \Psi:=\widetilde{\mathtt{D}}(\Id_{\mathbb{C}^4}+\mathtt{T})\,.
\end{equation*}
By using \eqref{total-diago} and \eqref{diago1wave} we have ${\mathtt{D}}\mathfrak{A}\widetilde{\mathtt{D}}=\Lambda$ (recall \eqref{diagonalissima}) modulo bounded remainders. Thus  we infer that
\begin{align}
\Phi\big( \mathfrak{A}+\mathfrak{B}\big)\Psi&= \Lambda+{\mathtt{D}}\mathfrak{B}\widetilde{\mathtt{D}}+\mathtt{T}\Lambda+\Lambda\mathtt{T}\label{sila1}\\
&+\mathtt{T}{\mathtt{D}}(\mathfrak{A}+\mathfrak{B})\widetilde{\mathtt{D}}\mathtt{T}+\mathtt{T}{\mathtt{D}}\mathfrak{B}\widetilde{\mathtt{D}}+ {\mathtt{D}}\mathfrak{B}\widetilde{\mathtt{D}} \mathtt{T}+R_0,\label{sila2}
\end{align}
for some bounded remainder $R_0$.
All the terms in \eqref{sila2} are bounded remainders, while the terms in \eqref{sila1} have positive order, in particular we need to see a cancellation  in ${\mathtt{D}}\mathfrak{B}\widetilde{\mathtt{D}}+\mathtt{T}\Lambda+\Lambda\mathtt{T}$. In its extended form such an expression takes the form
\begin{equation}\label{sila3}
\begin{aligned}
{\mathtt{D}}\mathfrak{B}\widetilde{\mathtt{D}}+\mathtt{T}\Lambda+\Lambda\mathtt{T}=&
\left(\begin{matrix}0& {\mathtt{D}}_{b}E\opbw(B_b(x,\xi))\widetilde{\mathtt{D}}_w  \\ {\mathtt{D}}_{w}E\opbw(B_w(x,\xi))\widetilde{\mathtt{D}}_b &0
 \end{matrix}\right)
 \\&\qquad \quad+
  \left(\begin{matrix} 0 & E\opbw(\lambda_b(x)\xi^2)\mathtt{T}_1 \\ E \opbw(\lambda_w(x)|\xi|)\mathtt{T}_2&0
 \end{matrix}\right)
 \\&
 \qquad\quad+  \left(\begin{matrix} 0 & \mathtt{T}_1E\opbw(\lambda_w(x)|\xi|) \\ \mathtt{T}_2E \opbw(\lambda_b(x)\xi^2)&0
 \end{matrix}\right)\,.
\end{aligned}
\end{equation}
We look for $\mathtt{T}_1, \mathtt{T}_2$ such that
\begin{equation*}
\begin{aligned}
{\mathtt{D}}_{b}E\opbw(B_b(x,\xi))\widetilde{\mathtt{D}}_w+E\opbw(\lambda_b(x)\xi^2)\circ\mathtt{T}_1=l.o.t
\\
{\mathtt{D}}_{w}E\opbw(B_w(x,\xi))\widetilde{\mathtt{D}}_b+\mathtt{T}_2\circ E\opbw(\lambda_b(x)\xi^2)=l.o.t\,,
\end{aligned}
\end{equation*}
where $``\emph{l.o.t.}''$ denotes paradifferential operators {with order at most zero.}

Now recall \eqref{parametrix-wave}, \eqref{parametrix} and \eqref{inverse2}. Then, if $\mathtt{T}_i=\opbw(T_i(x,\xi))$, we can use symbolic calculus (namely Proposition \ref{compo}) and obtain the equation on the matrices of symbols
\begin{equation*}
\begin{aligned}
k(x)S_b E B_b(x,\xi)S_w^{-1}&=-E\lambda_b(x)\xi^22T_1(x,\xi),\\
k^{-1}(x)S_wE B_w(x,\xi)S_b^{-1}&=-T_2(x,\xi)E\lambda_b(x)\xi^2.
\end{aligned}
\end{equation*}
Therefore it is enough to choose 
\begin{equation}\label{olio10}
T_1(x,\xi)=\frac{\psi(\xi)}{\xi^{2}}\frac{k(x)}{\lambda_b(x)}ES_b E B_b(x,\xi)S_w^{-1}\,,
\end{equation}
 and 
 \begin{equation}\label{olio11}
 T_2(x,\xi)=\frac{\psi(\xi)}{\xi^{2}}\frac{1}{\lambda_b(x)k(x)}S_wE B_w(x,\xi)S_b^{-1}E\,,
 \end{equation}
where $\psi$ is a function cutting-off the zero frequency as in \eqref{generatore-diago-ord-1}. 
By estimates 
\eqref{cliff} on $S_b$ and $\lambda_b$, the bounds
\eqref{cliffwave} on $S_w$ and $\lambda_w$,
 \eqref{figaro3} on $B_b$ and $B_{w}$, one can deduce the bound 
 \eqref{olio3} on $T_i$.
 Hence,
 by Theorems \ref{azione}, \ref{compo} the terms in \eqref{sila2} and \eqref{sila3}
 can be absorbed in the remainder $\mathcal{M}$ satisfying \eqref{olio5}.
The bound \eqref{olio}
follows by \eqref{parametrix-wave}-\eqref{totalMappaWave}
and \eqref{parametrix}-\eqref{totalMappa}.
Similarly one deduces the estimates \eqref{stimephipsi}.
Finally, reasoning as in the proof of \eqref{cliffwave200},
using formul\ae\,
\eqref{olio10}-\eqref{olio11} and hypotheses
\eqref{palla-beam}-\eqref{palla-wave}
one can show that 
\[
|\pa_{t}T_{i}|_{-\frac{3}{2},s_0,\alpha}\lesssim_{s_0, \alpha}C_{R} \qquad \forall \alpha\in\mathbb{N}\,.
\]
This implies together with estimates \eqref{cliffwave200} and Theorem \ref{azione}
the bound \eqref{stimatempoaltempo}.
This concludes the proof.
\end{proof}
\subsection{Equivalence of the norms}
We introduce the operator, for $s\geq0$, 
\[
\mathfrak{L}_{2s}=\mathfrak{L}_{2s}(\widetilde{V}):=\left(\begin{matrix} 
\opbw(\Id_{\C^2}\lambda_{b}^s\,\xi^{2s}) & 0 \\ 0 &
\opbw(\Id_{\C^2}\lambda_{w}^{2s}\,\x^{2s})
\end{matrix}\right)
\]
and we set (recall \eqref{scalarprod4x4})
\begin{equation}\label{regnodinapoli}
|V|^{2}_{\widetilde{V},s}:=\langle \mathfrak{L}_{2s}(\widetilde{V})
\Phi(\widetilde{V})V,\Phi(\widetilde{V})V\rangle, \qquad V:=(Z, W) \,.
\end{equation}

The next proposition is fundamental for the proof of the main result. 
We establish the \emph{almost} equivalence 
between the aforementioned norms and the usual ones.

\begin{proposition}\label{equivalenzatotale}
Recall \eqref{palla-beam}-\eqref{palla-wave}, we have for all $\s\geq 0$ 
\begin{align}
C_r^{-1} (\|V\|_{H^{\s}}^2-\|V\|_{H^{-2}}^2)&\le |V|_{\widetilde{V},\s}^2\le C_r\|V\|^2_{H^{\s}}\,,
\label{normabeamequiv}
\\
C_r^{-1}(\|Z\|^{2}_{H^{\sigma}}-\|Z\|^{2}_{H^{-2}})&\le 
(\opbw(\lambda_b^{\sigma}\xi^{2\sigma})\mathtt{D}_{b}
Z,\mathtt{D}_{b}Z)_{L^2\times L^2},
\label{norminabeam}
\\
C_r^{-1}(\|W\|_{H^{\sigma}}^2-\|W\|_{H^{-1}}^2)&\leq 
(\opbw(\lambda_w^{2\s} \xi^{2\s})\mathtt{D}_wW,\mathtt{D}_wW)_{L^2\times L^2}\,.
\label{norminawave}
\end{align}
\end{proposition}
\begin{proof}
We start by proving \eqref{normabeamequiv}.
The upper bound in \eqref{normabeamequiv} follows by using \eqref{stimephipsi}, indeed 
\begin{equation*}
|V|^2_{\widetilde{V}, \s}\leq \|\mathfrak{L}_{2 \s}(\widetilde{V})\Phi(\widetilde{V})V\|_{H^{-\sigma}}\|\Phi(\widetilde{V})V\|_{H^{\sigma}}\lesssim_{\s, r}\|V\|_{H^{\sigma}}^2.
\end{equation*}
We now focus on the lower bound in \eqref{normabeamequiv}. Let $\delta>0$ such that $s_0-\delta>1/2$ . 
For any $\sigma\in\R$ set
\begin{equation*}\widetilde{\Lambda}^{\sigma}:=\left(\begin{matrix} 
\opbw(\Id_{{\C}^2}\lambda_{b}^{\sigma/2}) & 0 \\ 0 &
\opbw(\Id_{\C^2}\lambda_{w}^{\sigma})
\end{matrix}\right).
\end{equation*}
By using Theorem \ref{compo} with ${s}_0-\delta$ 
instead of $s_0$ and $\rho=\delta$, and recalling 
\eqref{cliff} and \eqref{cliffwave},
 we obtain 
\begin{equation}\label{mannaggia1}
\begin{aligned}
\widetilde{\Lambda}^{\sigma}\opbw(\xi^{2\sigma}\Id_{\C^4})\widetilde{\Lambda}^{\sigma}=\mathfrak{L}_{2\sigma}(\widetilde{V})+\mathcal{R}_1^{2\sigma-\delta}
\end{aligned}\end{equation}
where 
\begin{equation}\label{salcazzo}
\|\mathcal{R}_1^{2\sigma-\delta}f\|_{H^{\bar{\sigma}-2\sigma+\delta}}\le 
C_r \|f\|_{H^{\bar{\sigma}}} \quad \forall\bar{\sigma}\geq 0
\end{equation}
and any function $f$ in $H^{\bar\sigma}$.  Analogously we prove that
\begin{equation}\label{mannaggia2}
\widetilde{\Lambda}^{-\sigma}\Psi(\widetilde{V})\widetilde{\Lambda}^{\sigma}\Phi(\widetilde{V})=\Id_{\C^4}+\mathcal{R}_2^{-\delta}\,,
\qquad
\|\mathcal{R}_2^{-\delta}\|_{\mathcal{L}(H^{\bar{\s}-\delta};H^{\bar{\s}})}\le C_{r}\,.
\end{equation}

Owing to \eqref{mannaggia1} and \eqref{mannaggia2} we infer that
\begin{equation*}
\|V\|_{H^{\sigma}}^2\stackrel{\eqref{mannaggia2}}{\le}\|\widetilde{\Lambda}^{-\sigma}\Psi\widetilde{\Lambda}^{\sigma}\Phi V\|_{H^{\sigma}}\stackrel{(*)}{\le}C_r(\|\widetilde{\Lambda}^{\sigma}\Phi(\widetilde{V}) V\|_{H^{\sigma}}^2+\|V\|^2_{H^{\sigma-\delta}}),
\end{equation*}
where in $(*)$ we have used the boundedness of the operator $\widetilde{\Lambda}^{-\sigma}\Psi$ given by Proposition \ref{prop:costruzioneMappa}. 
We can continue the chain of inequalities and obtain (recall \eqref{regnodinapoli})
\begin{equation*}
\begin{aligned}
\|\widetilde{\Lambda}^{\sigma}\Phi(\widetilde{V}) V\|_{H^{\sigma}}^2&=\langle \opbw(\langle\xi\rangle^{\sigma})\widetilde{\Lambda}^{\sigma}\Phi(\widetilde{V})V, \opbw(\langle\xi\rangle^{\sigma})\widetilde{\Lambda}^{\sigma}\Phi(\widetilde{V}) {V}\rangle\\
&=\langle  \widetilde{\Lambda}^{\sigma}\opbw(\langle\xi\rangle^{2\sigma})\widetilde{\Lambda}^{\sigma}\Phi(\widetilde{V})V,\Phi(\widetilde{V})V
\rangle\\
&\stackrel{\eqref{mannaggia1}}{\leq} C_r (|V |_{\widetilde{V},\sigma}^2+\langle\mathcal{R}_1^{2\sigma-\delta}\Phi(\widetilde{V})V,\Phi(\widetilde{V})V\rangle)\\
&\leq C_r (|V|_{\widetilde{V},\sigma}^2+\|V\|^2_{H^{\sigma-\delta/2}}).
\end{aligned}
\end{equation*}
We explain the penultimate  inequality,  we shall use the estimates {\eqref{salcazzo}} on  
$\mathcal{R}_1^{2\sigma-\delta}$ {with $\bar\sigma=\sigma-\delta/2$} and 
the estimate  \eqref{stimephipsi} to obtain 
\begin{equation*}
\begin{aligned}
\langle\mathcal{R}_1^{2\sigma-\delta}\Phi(\widetilde{V})V,\Phi(\widetilde{V})V\rangle_{L^2}
&\leq 
\|\mathcal{R}_1^{2\sigma-\delta}\Phi(\widetilde{V})V\|_{H^{\delta/2-\sigma}} 
\|\Phi(\widetilde{V})V\|_{H^{\sigma-\delta/2}}
\\&\leq 
C_r\|\Phi(\widetilde{V})V\|_{H^{\sigma-\delta/2}}^2
\leq C_r\|V\|_{H^{\sigma-\delta/2}}^2\,.
\end{aligned}
\end{equation*}

Resuming we proved the inequality
\begin{equation*}
\|V\|_{H^{\sigma}}^2\leq C_r(|V|_{\widetilde{V},\sigma}^2+\|V\|_{H^{\sigma-\delta/2}}^2)\,.
\end{equation*}
In order to control $\|V\|_{H^{\sigma-\delta/2}}^2$ we use the interpolation inequality 
\eqref{intorpolostima}
with $s_1\rightsquigarrow-2$, $s_2\rightsquigarrow\s$, 
$s\rightsquigarrow\s-\delta/2$
and  Young inequality\footnote{$ab\leq a^p/p+b^q/q$ for any $a,b\geq0$ and $p,q>1$ 
such that $1/p+1/q=1$}
we obtain
\begin{equation*}
\begin{aligned}
\|V\|_{H^{\sigma-\delta/2}}^2&\leq 
(\|V\|^2_{H^{-2}})^{\frac{\delta}{2(\sigma+2)}}
(\|V\|_{H^{\sigma}}^2)^{\frac{2(\sigma+2)-\delta}{2(\sigma+2)}}
\\&
\leq 
\tfrac{\delta}{2(\sigma+2)}\|V\|_{H^{-2}}^2\tau^{-\frac{2(\sigma+2)}{\delta}}
+\tfrac{2(\sigma+2)-\delta}{2(\sigma+2)}
\tau^{\frac{2(\sigma+2)-\delta}{2(\sigma+2)}}
\|V\|_{H^{\sigma}}^2\,.
\end{aligned}
\end{equation*}
We then conclude by choosing $\tau$ small enough.

In order to prove \eqref{norminabeam} (resp. \eqref{norminawave})
one can follow 	almost word by word the reasoning above by substituting
$\Phi\rightsquigarrow \mathtt{D}_{b}$ and  $\Psi\rightsquigarrow \widetilde{\mathtt{D}}_{b}$ 
(resp. $\Phi\rightsquigarrow \mathtt{D}_{w}$ and  $\Psi\rightsquigarrow \widetilde{\mathtt{D}}_{w}$ ).
\end{proof}
\subsection{Garding-type inequality}
Let us define
\begin{equation}\label{deltone}
{\bf \Delta}:=\left(\begin{matrix}\Id_{\C^2}\pa_{x}^{4} & 0 \\ 0& -\Id_{\C^2}\pa_{x}^{2}\end{matrix}\right)
\end{equation}
We prove the following.
\begin{lemma}\label{Garding-deltone}
Let $\sigma\geq 0$.
The following inequalities hold true
\begin{equation}\label{supergarding}
\begin{aligned}
\langle
\mathfrak{L}_{2\s}(\widetilde{V})\Phi(\widetilde{V}){\bf \Delta}V,\Phi(\widetilde{V})V\rangle
&\geq 
(\|Z\|_{H^{\s+2}}^2+\|W\|_{H^{\s+1}}^{2}) - \,C_r \|V\|^2_{H^{\s}}\,,
\\
\langle\mathfrak{L}_{2\s}(\widetilde{V})\Phi(\widetilde{V})V,\Phi(\widetilde{V}){\bf \Delta}V\rangle
&\geq 
(\|Z\|_{H^{\s+2}}^2+\|W\|_{H^{\s+1}}^{2})- \,C_r \|V\|^2_{H^{\s}}\,\,,
\end{aligned}
\end{equation}
for some $C_r>1$ and for any $V=(Z, W)\in \mathcal{H}^{\s+2}\times\mathcal{H}^{\s+1}$.
\end{lemma}

\begin{proof}
We only prove the first inequality in \eqref{supergarding} being the second one similar.
Recall the expression  \eqref{expaMappa} of $\Phi=\mathtt{D}+\mathtt{D}\mathtt{T}$ 
with  the matrices of operators $\mathtt{D}$ and $\mathtt{T}$  
in \eqref{matricitt} respectively of order $0$ 
(see \eqref{totalMappa}, \eqref{totalMappaWave}) and $-3/2$.
Then one has 
\begin{align}
\langle \mathfrak{L}_{2\s}&(\widetilde{V})\Phi(\widetilde{V}){\bf \Delta}V, \Phi(\widetilde{V})V\rangle
= \langle \mathfrak{L}_{2\s}(\widetilde{V})\mathtt{D}{\bf \Delta}V, \mathtt{D}V\rangle\label{gardingalta}
\\&
+\langle \mathfrak{L}_{2\s}(\widetilde{V}) \mathtt{D}\mathtt{T}{\bf\Delta}V,\mathtt{D}V\rangle
+\langle \mathfrak{L}_{2\s}(\widetilde{V}) \mathtt{D}{\bf\Delta}V,\mathtt{D}\mathtt{T}V\rangle
+\langle \mathfrak{L}_{2\s}(\widetilde{V}) \mathtt{D}\mathtt{T}{\bf\Delta}V,\mathtt{D}
\mathtt{T}V\rangle\label{gardingbassa}\,.
\end{align}
We claim that \eqref{gardingbassa} is bounded from above by 
$C_r(\|Z\|_{H^{\sigma+7/4}}^2+\|W\|_{H^{\sigma+3/4}}^2)$. 
Consider for instance the first addendum, it equals to
\begin{equation}\label{bullo}
\begin{aligned}
\langle \mathfrak{L}_{2\s}(\widetilde{V}) \mathtt{D}\mathtt{T}{\bf\Delta}V,\mathtt{D}V\rangle=
&
-\underbrace{(\opbw(\lambda_b^{\sigma}\xi^{2\sigma})
\mathtt{D}_b\mathtt{T}_1\partial_x^2W,\mathtt{D}_bZ)_{L^2\times L^2}}_{(I)}
\\&
+\underbrace{(\opbw(\lambda_w^{2\s}\xi^{2\s}) 
\mathtt{D}_w\mathtt{T}_2\partial_x^4Z,\mathtt{D}_wW)_{L^2\times L^2}}_{(II)}\,.
\end{aligned}
\end{equation}
We can write
\begin{equation*}
(I)=(\opbw(\langle\xi\rangle^{-\sigma-\frac32})\opbw(\lambda_b^{\sigma}\xi^{2\sigma})\mathtt{D}_b\mathtt{T}_1\partial_x^2W,\opbw(\langle\xi\rangle^{\sigma+\frac32})\mathtt{D}_bZ)_{L^2\times L^2}.
\end{equation*}
By using the Cauchy-Schwarz inequality, the action Theorem \ref{azione} and the bounds \eqref{olio3}, \eqref{olio} we may bound the above quantity by
$C_r (\|Z\|_{H^{\sigma+3/2}}^2+\|W\|_{H^{\s-1}}^2)$. Reasoning in the same way, we have
\begin{equation*}
\begin{aligned}
(II)=(\opbw(\langle\xi\rangle^{-\s-\frac34})\opbw(\lambda_w^{2\s}\xi^{2\s}) \mathtt{D}_w\mathtt{T}_2\partial_x^4Z,\opbw(\langle\xi\rangle^{\s+\frac34})\mathtt{D}_wW)_{L^2\times L^2}\leq\\
C_r ( \|Z\|^2_{H^{\s+7/4}}+\|W\|_{H^{\s+3/4}}^2).
\end{aligned}\end{equation*}
Therefore we proved the claimed estimate on \eqref{bullo}.
The other two addends in \eqref{gardingbassa} may be handled in a similar way, actually the third one has better estimates thanks to the presence of $\mathtt{T}$ on both sides of the scalar product. \\
Consider  the r.h.s. of \eqref{gardingalta}. We have
\begin{equation*}\begin{aligned}
\langle \mathfrak{L}_{2\s}(\widetilde{V})\mathtt{D}{\bf \Delta}V, \mathtt{D}V\rangle=&\underbrace{(\opbw(\lambda_b^{\s}\xi^{2\s})\mathtt{D}_b\opbw(\xi^4)Z,\mathtt{D}_bZ)_{L^2\times L^2}}_{(III)}\\
+&\underbrace{(\opbw(\lambda_b^{2\s}\xi^{2\s})\mathtt{D}_w\opbw(\xi^2)W,\mathtt{D}_wW)_{L^2\times L^2}}_{(IV)}.
\end{aligned}\end{equation*}
By using symbolic calculus, i.e. Theorem \ref{compo}, 
and the self-adjointness of $\opbw({\xi^2})$ we have 
\begin{align*}
(III)
=\underbrace{(\opbw(\lambda_b^{\sigma}\xi^{2\sigma})\mathtt{D}_{b}
\opbw({\xi^2})Z,\mathtt{D}_{b}\opbw({\xi^2})Z) _{L^2\times L^2}}_{(V)}
+( R_{2\sigma+3}Z,\mathtt{D}_{b}Z)_{L^2\times L^2},
\end{align*}
where $R_{2\sigma+3}$ is a bounded linear operator from 
$\mathcal{H}^{s+2\sigma+3}$ to 
$\mathcal{H}^{s}$ for $s\in\R$. 
By \eqref{norminabeam}, we have
\begin{equation*}
(V)
\geq 
C_r(\|Z\|^{2}_{H^{\sigma+2}}-\|Z\|^{2}_{L^2})\,.
\end{equation*}
We now estimate from above 
$( R_{2\sigma+3}Z,\mathtt{D}_{b}Z)_{L^2\times L^2}$. 
We have
\begin{equation*}
\begin{aligned}
( R_{2\sigma+3}Z,\mathtt{D}_{b}Z)_{L^2\times L^2}&=
(\opbw(\langle\xi\rangle^{-\sigma-3/2})\mathcal{R}_{2\sigma+3}Z,
\opbw(\langle\xi\rangle^{\sigma+3/2})\mathtt{D}_{b}Z )_{L^2\times L^2}
\\&
\leq C_r \|Z\|^2_{H^{\sigma+3/2}}\,.
\end{aligned}
\end{equation*}
Up to now we proved that
\begin{equation*}
(III)
\geq 
C_r(\|Z\|^2_{H^{\sigma+2}}-\|Z\|_{H^{\sigma+3/2}}^2)\,.
\end{equation*}
Analogously, using \eqref{norminawave}, one proves that 
\begin{align*}
(IV)\geq C_r(\|W\|_{H^{\sigma+1}}^2-\|W\|_{H^{\sigma+1/2}}^2)
\end{align*}
Putting together all these estimates on \eqref{gardingbassa} and on the r.h.s. of \eqref{gardingalta}, we infer that
\begin{equation*}
\begin{aligned}
\langle \mathfrak{L}_{2\s}&(\widetilde{V})\Phi(\widetilde{V}){\bf \Delta}V, \Phi(\widetilde{V})V\rangle
\\
&\geq C_{r}(\|Z\|_{H^{\s+2}}^2+\|W\|_{H^{\s+1}}^2-\|Z\|_{H^{\s+3/2}}^2-\|W\|_{H^{\s+1/2}}^2-\|Z\|_{H^{\s+7/4}}^2-\|W\|_{H^{\s+3/4}}^2)
\\
&\geq {C}_{r}(\|Z\|_{H^{\s+2}}^2+\|W\|_{H^{\s+1}}^2-2 \|Z\|_{H^{\s+7/4}}^2-2 \|W\|_{H^{\s+3/4}}^2).
\end{aligned}
\end{equation*}
We use the interpolation inequalities 
\[
\|Z\|_{H^{\sigma+7/4}}\leq
\eta_1\|Z\|_{H^{\sigma+2}}+\frac{1}{\eta_1}\|Z\|_{H^{\sigma}}\,,
\qquad 
\|W\|_{H^{\sigma+3/4}}\leq \eta_2\|W\|_{H^{\sigma+1}}+\frac{1}{\eta_2}\|W\|_{H^{\sigma}}\,,
\] 
 by choosing $\eta_1=\eta_2=\frac{C_r-1}{2 C_r}$  we obtain the \eqref{supergarding}.
 \end{proof}

\section{Linear well posedness}\label{sec:linear}

We shall study the well posedness of the problem
\eqref{linprob}, with $\widetilde{V}$ satisfying \eqref{palla-beam}-\eqref{palla-wave}.
We need to consider the associated 
regularized, homogeneous  equation.  More precisely,
following  
\cite{iandolikdv, KPV2}, 
 we consider the following parabolic regularization  
\begin{align}
\partial_tV^{\epsilon}&=\big(\mathfrak{A}(\widetilde{V})+
\mathfrak{B}(\widetilde{V})\big)V^{\epsilon}-\epsilon{\bf \Delta}V^{\epsilon}\,,
\qquad 
V^{\epsilon}:=\vect{Z^{\epsilon}}{W^{\epsilon}}\,.
\label{beam-smoo}
\end{align}
The following holds.
\begin{lemma}\label{lem:lwp1reg}
Assume \eqref{palla-beam}-\eqref{palla-wave} and let $\sigma\geq 0$. For any $\epsilon>0$ there exists a time 
$T_{\epsilon}>0$ such that the following holds true. 
For any initial datum ${V_0:=}(Z_0,W_0):= (Z(0),W(0))\in\mathcal{H}^{\sigma}\times \mathcal{H}^{\sigma}$, 
there exists a unique solution $V^{\epsilon}(t)$ of the equation \eqref{beam-smoo} 
that belongs to the space 
$C^0([0,T_{\epsilon}); \cH^{\sigma}\times\cH^{\s})\cap C^1([0,T_{\epsilon});\cH^{\s-2}\times\cH^{\s-1})$. 
\end{lemma}

\begin{proof}
Let $V:=(Z, W) $ and consider the map 
\[
\begin{aligned}
\mathfrak{F}: & V \longmapsto \mathfrak{F} V:= 
e^{-\epsilon t{\bf \Delta}}V_0+\int_0^t F^{t,t'}(\widetilde{V})V(t')dt'\,,
\\
F^{t,t'}(\widetilde{V})&:=e^{-\epsilon(t-t'){\bf \Delta}}
\big(\mathfrak{A}(\widetilde{V})+
\mathfrak{B}(\widetilde{V})\big)
\end{aligned}
\] 
where $e^{-\epsilon t{\bf \Delta}}$ is the operator defined as (recall \eqref{deltone})
\[
e^{-\epsilon t{\bf \Delta}}=\left(
\begin{matrix}
\Id_{\C^2}e^{-\epsilon t\pa_{x}^{4}} & 0 \\ 0 & \Id_{\C^2}e^{\epsilon t\pa_{x}^{2}}
\end{matrix}
\right)\,,
\qquad 
e^{-\partial^4_x} e^{i\xi x} = e^{-\xi^4} e^{i\xi x},\quad e^{\partial^2_x} e^{i\xi x} = e^{-\xi^2} e^{i\xi x}\,,\quad \forall\, \x\in\Z\,.
\]
With this notation, and recalling \eqref{matricitotali}, we shall write 
\[
F^{t,t'}(\widetilde{V})V(t')=\left(
\begin{matrix}
\Id_{\C^2}e^{-\epsilon (t-t')\pa_{x}^{4}} G_1(\widetilde{V}(t'))V(t')
\\
\Id_{\C^2}e^{\epsilon (t-t')t\pa_{x}^{2}} G_2(\widetilde{V}(t'))V(t')
\end{matrix}
\right)
\]
where 
\[
\begin{aligned}
G_1(\widetilde{V})V&=
-\ii E\opbw(A_b(x,\x))Z
-\ii E\opbw(B_b( \widetilde{V}; x,\x))W,
\\
G_2(\widetilde{V})V&=
-\ii E\opbw(A_w(\widetilde{V}; x,\x))W
-\ii E\opbw(B_w(\widetilde{V};x,\x))Z\,.
\end{aligned}
\]
First of all, by action Theorem \ref{azione} and the estimates on the seminorms of the symbols 
\eqref{opside2}, \eqref{figaro3}, \eqref{figaro4}, we deduce 
\begin{equation}\label{latte4}
\begin{aligned}
\|G_1(\widetilde{V})V\|_{L^{\infty}H^{\s-2}}
&\lesssim_{s}\|V\|_{L^{\infty}H^\s}\,,
\\
\|G_2(\widetilde{V})V\|_{L^{\infty}H^{\s-1}}
&\lesssim_{s}\|V\|_{L^{\infty}H^\s}\,.
\end{aligned}
\end{equation}
Moreover we claim that the following bounds hold true:
\begin{equation}\label{accalde}
\begin{aligned}
\|e^{{-\epsilon}t\partial_x^4}f\|_{H^{\sigma}}\,, \|e^{{\epsilon}t\partial_x^2}f\|_{H^{\sigma}} & \leq \|f\|_{H^{\sigma}}
\\
\left\|\int_0^t e^{-{\epsilon(t-t')}\partial_x^4}f(t',\cdot)dt' \right\|_{H^{\sigma}} & 
\lesssim t^{\frac12}\epsilon^{-\frac12}\|f\|_{L^{\infty}H^{\sigma-2}}\,,
\\
\left\|\int_0^t e^{{\epsilon(t-t')}\partial_x^2}f(t',\cdot)dt' \right\|_{H^{\sigma}} & 
\lesssim t^{\frac34}\epsilon^{-\frac14}\|f\|_{L^{\infty}H^{\sigma-1}}\,,
\end{aligned}
\end{equation}
for any $f\in L^{\infty}([0,T_{\epsilon}); \cH^{\sigma})$. 
We prove the second estimate in \eqref{accalde}, the first one being straightforward and the third very similar to the second.
 We have, using Minkowski inequality, the following chain of inequalities 
\begin{align*}
\|\int_0^t e^{-\epsilon(t-t')\partial_x^4}&f(t',\cdot)dt'\|_{H^{\sigma}}\leq\int_0^t 
\|e^{-\epsilon(t-t')\partial_x^4}f(t',\cdot)\|_{H^{\sigma}}dt'
\\
&\leq \int_0^t\Big(\sum_{\xi}e^{-2\epsilon(t-t')\xi^4}\xi^{2\sigma}|\hat{f}(t',\xi)|^{2}\Big)^{1/2}dt'\\
&\leq C\|f\|_{L^{\infty}H^{\sigma-2}}\epsilon^{-1/2}\int_0^t(t-t')^{-1/2}dt'\leq C\|f\|_{L^{\infty}H^{\sigma-2}}\epsilon^{-1/2}t^{1/2},
\end{align*}
where we have also used  the fact that the function 
$ye^{-y}$ is bounded for $y\geq0$.

\noindent
By means of \eqref{accalde} one proves that
that the map $\mathfrak{F}$ defines  a contraction, on a suitable ball 
of the Banach space $C^0([0,T_{\epsilon}); \mathcal{H}^{\sigma}\times \mathcal{H}^{\sigma})$
provided that $T_{\epsilon}\epsilon^{-1}$ is small enough.
Hence the thesis follows.
\end{proof}

To remove the $\epsilon$-dependence on the time 
$T_{\epsilon}$ and therefore pass to the limit for $\epsilon\rightarrow 0$, we need to prove
 \emph{a priori} estimates on the equation \eqref{beam-smoo} {uniformly in the parameter $\epsilon$}.
 To do this we shall use the \emph{modified energy}
 defined in \eqref{regnodinapoli}.
 \begin{proposition}\label{stima-energia-totale}
Assume \eqref{palla-beam}-\eqref{palla-wave}, let $\sigma\geq0$ and consider \eqref{beam-smoo} 
with initial condition $V(0)=V_0\in\cH^{\sigma}\times\cH^{\s}$. {Then the solution $V^{\epsilon}=V^{\epsilon}(t)$, provided by Lemma \ref{lem:lwp1reg}, must satisfy the following estimate}
\begin{equation}\label{energy-beam}
\|V^{\epsilon}(t)\|_{H^{\sigma}}^2\leq C_r\|V_0\|_{H^{\sigma}}^2
+C_R\int_0^t\|V^{\epsilon}(\tau)\|_{H^{\sigma}}^2\,,
\end{equation}
for some $C_{r}$ and $C_R$ depending on \eqref{palla-beam}-\eqref{palla-wave} and $\sigma$,
which in turn implies
\begin{equation}\label{buona-energy-beam}
\|V^{\epsilon}(t)\|_{H^{\sigma}} \leq C_r\exp(C_Rt)\|V_0\|_{H^{\sigma}}\, \qquad \forall \epsilon\geq 0.
\end{equation}
\end{proposition}
\begin{remark}\label{rem:unifeps}
Since the estimate \eqref{buona-energy-beam} is independent of $\epsilon$ we may repeat the proof of Lemma \ref{lem:lwp1reg} on the intervals $[T_{\epsilon},T_{2\epsilon}]$,  $[T_{2\epsilon},T_{3\epsilon}]$... and eventually obtain a common time of existence $T>0$ which is independent of $\epsilon$.
\end{remark}
\begin{proof}

We take the time derivative of the modified energy \eqref{regnodinapoli} along the solutions of the system \eqref{beam-smoo}. We have
\begin{align}
\tfrac{d}{dt}|V^{\epsilon}|_{\widetilde{V},\s}^2=&\langle \tfrac{d}{dt} \big(\mathfrak{L}_{2\s}(\widetilde{V})\big)\Phi(\widetilde{V})V^{\epsilon},\Phi(\widetilde{V})V^{\epsilon}\rangle\label{dtL}
\\&
+\langle \mathfrak{L}_{2\s}(\widetilde{V})\tfrac{d}{dt} \big(\Phi(\widetilde{V})\big)V^{\epsilon},\Phi(\widetilde{V})V^{\epsilon}\rangle+\langle \mathfrak{L}_{2\s}(\widetilde{V})\Phi(\widetilde{V})V^{\epsilon},\tfrac{d}{dt} \big(\Phi(\widetilde{V})\big)V^{\epsilon}\rangle\label{dtphi}
\\&
+\langle \mathfrak{L}_{2\s}(\widetilde{V})\Phi(\widetilde{V})
\tfrac{d}{dt}V^{\epsilon}, \Phi(\widetilde{V})V^{\epsilon}\rangle
+\langle \mathfrak{L}_{2\s}(\widetilde{V})\Phi(\widetilde{V})V^{\epsilon}, 
\Phi(\widetilde{V})\tfrac{d}{dt}V^{\epsilon}\rangle\label{cancelli}\,.
\end{align}
By using the action Theorem \ref{azione}, the estimate on $\pa_{t}\Phi$ in 
\eqref{stimatempoaltempo}, the estimate on $\pa_{t}\lambda_w$ in 
\eqref{cliffwave200} to bound $\pa_{t}\mathfrak{L}_{2\s}(\widetilde{V})$,
one may bound terms 
in \eqref{dtL} and \eqref{dtphi} from above 
by $C_R\|V^{\epsilon}\|_{H^{\sigma}}^2$.

 In \eqref{cancelli} we have to see a cancellation. By using the equation \eqref{beam-smoo} we may rewrite \eqref{cancelli} as 
\begin{align}
&\langle\mathfrak{L}_{2\s}(\widetilde{V})\Phi(\widetilde{V})[\mathfrak{A}(\widetilde{V})+\mathfrak{B}(\widetilde{V})]V^{\epsilon}, \Phi(\widetilde{V})V^{\epsilon}\rangle+\langle\mathfrak{L}_{2\s}(\widetilde{V})\Phi(\widetilde{V})V^{\epsilon}, \Phi(\widetilde{V})[\mathfrak{A}(\widetilde{V})+\mathfrak{B}(\widetilde{V})]V^{\epsilon}\rangle\label{cancelli2}\\ 
&-\epsilon \langle\mathfrak{L}_{2\s}(\widetilde{V})\Phi(\widetilde{V}){\bf \Delta}V^{\epsilon}, \Phi(\widetilde{V})V^{\epsilon}\rangle
-\epsilon \langle\mathfrak{L}_{2\s}(\widetilde{V})\Phi(\widetilde{V})V^{\epsilon}, \Phi(\widetilde{V}){\bf\Delta}V^{\epsilon}\rangle\label{usogarding}.
\end{align}
The expression \eqref{usogarding} is bounded from above by $C_r\|V^{\epsilon}\|^2_{H^{\s}}$ thanks to \eqref{supergarding}. Concerning the expression \eqref{cancelli2}, we use that $\Psi\circ \Phi-\Id_{\mathbb{C}^4}$
belongs to
 $\mathcal{L}\big(\mathcal{H}^{\s}\times\mathcal{H}^{\s};\mathcal{H}^{\s+2}\times\mathcal{H}^{\s+2}\big)$ 
 thanks to Proposition \ref{prop:costruzioneMappa}  (see \eqref{stimephipsi})
 and we rewrite it (modulo terms bounded by $C_R\|V^{\epsilon}\|^2_{H^{\s}}$) as
 \begin{equation*}
 \begin{aligned}
 &\langle\mathfrak{L}_{2\s}(\widetilde{V})\Phi(\widetilde{V})[\mathfrak{A}(\widetilde{V})+\mathfrak{B}(\widetilde{V})]\Psi(\widetilde{V})\Phi(\widetilde{V})V^{\epsilon}, \Phi(\widetilde{V})V^{\epsilon}\rangle+\\
& \langle\mathfrak{L}_{2\s}(\widetilde{V})\Phi(\widetilde{V})V^{\epsilon}, \Phi(\widetilde{V})[\mathfrak{A}(\widetilde{V})+\mathfrak{B}(\widetilde{V})]\Psi(\widetilde{V})\Phi(\widetilde{V})V^{\epsilon}\rangle.
 \end{aligned}
 \end{equation*}
 
 
 At this point we are in position to use \eqref{fessemamt} and obtain
  \begin{align*}
 \langle\mathfrak{L}_{2\s}(\widetilde{V})\Lambda\Phi(\widetilde{V})V^{\epsilon}, \Phi(\widetilde{V})V^{\epsilon}\rangle+
 \langle\mathfrak{L}_{2\s}(\widetilde{V})\Phi(\widetilde{V})V^{\epsilon},\Lambda\Phi(\widetilde{V})V^{\epsilon}\rangle,
 \end{align*}
 modulo remainders bounded by $C_R\|V^{\epsilon}\|^2_{H^{\s}}$, and where $\Lambda$ is defined in \eqref{diagonalissima}. Using the skew-selfadjoint character of $\Lambda$, the above expression is equal to 
 \begin{align*}
 \langle\Big[\mathfrak{L}_{2\s}(\widetilde{V}),\Lambda\Big]\Phi(\widetilde{V})V^{\epsilon}, \Phi(\widetilde{V})V^{\epsilon}\rangle.
 \end{align*}
 Because of the choice of the operator $\mathfrak{L}_{2\s}$ we have that the commutator $\Big[\mathfrak{L}_{2\s}(\widetilde{V}),\Lambda\Big]$ equals to zero, modulo linear operators in $\mathcal{L}(\cH^{2\sigma}\times \cH^{2\sigma}, \cH^0 \times \cH^0)$ whose operator norm is bounded by $C_R$, indeed we have
 \begin{equation*}
 \Big[\mathfrak{L}_{2\s}(\widetilde{V}),\Lambda\Big]=\Big[\Id_{\C^{2}}\opbw(\lambda_b^{\s}\xi^{2\s}),E\opbw(\lambda_b\xi^{2})\Big]
 +\Big[\Id_{\C^{2}}\opbw(\lambda_w^{2\s}\xi^{2\s}), E\opbw(\lambda_w|\xi|)\Big]\,,
 \end{equation*}
from which one uses Theorem \ref{compo} and the fact that the Poisson bracket of a symbol and a power of it equals to zero.

\noindent
Up to now we proved that 
\begin{equation*}
\frac{d}{dt}|V^{\epsilon}|_{\widetilde{V},\s}^2\leq C_R\|V^{\epsilon}\|^2_{H^{\s}}.
\end{equation*}
Integrating in time we obtain
\begin{equation}\label{quasi-energia}
|V^{\epsilon}(t)|_{\widetilde{V},\s}^2\leq |V^{\epsilon}(0) |^2_{\widetilde{V},\s}
+C_R\int_0^t \|V^{\epsilon}(\tau)\|_{H^{\sigma}}^2d\tau \stackrel{\eqref{normabeamequiv}}{\leq} C_r \|V^{\epsilon}(0)\|^2_{H^{\sigma}}+C_R\int_0^t \|V^{\epsilon}(\tau)\|_{H^{\sigma}}^2d\tau\,.
\end{equation}
On the other hand, recalling \eqref{normabeamequiv}, we have 
\begin{equation}
|V^{\epsilon}|_{\widetilde{V},\sigma}^2\geq C^{-1}_r\|V^{\epsilon}\|^{2}_{H^{\sigma}}
-C^{-1}_r\|V^{\epsilon}\|^{2}_{H^{-2}}.
\end{equation} 
We may bound the $H^{-2}$ norm 
by just loosing derivatives as follows:
\begin{align*}
\frac{d}{dt}\|V^{\epsilon}\|_{H^{-2}}^2\leq 2\langle \opbw(\langle\x\rangle^{-4})\partial_t V^{\epsilon}, V^{\epsilon}\rangle\leq 2\|V^{\epsilon}\|_{H^0}^2\leq 2\|V^{\epsilon}\|_{H^{\sigma}}^2,
\end{align*}
since $\sigma\geq 0$. Integrating the above inequality we get 
\begin{equation}\label{fine-energia}
\|V^{\epsilon}\|_{H^{-2}}^2\leq \|V^{\epsilon}(0)\|_{H^{-2}}^2
+2\int_0^t\|V^{\epsilon}(\tau)\|_{H^{\sigma}}^2d\tau \qquad \mathrm{for}\,\,\sigma\geq 0\,.
\end{equation}
Putting together \eqref{quasi-energia} and \eqref{fine-energia} 
we eventually deduce \eqref{energy-beam}. 
The \eqref{buona-energy-beam} is just an application of Gronwall lemma.
\end{proof}
We are now in position to state a proposition in which we prove the linear local well posedness of the linearized problem.
\begin{proposition}\label{LWP}
Consider $\widetilde{V}=(\widetilde{Z},\widetilde{W})$ satisfying \eqref{palla-beam}-\eqref{palla-wave} for some $T>0$. Let $\sigma\geq 0$ and $t\mapsto \mathfrak{R}(t)\in C^0([0,T];\cH^{\sigma}\times\cH^{\sigma})$. 
There exists a unique solution $V\in C^0([0,T) ; \cH^{\sigma}\times\cH^{\s})\cap C^1([0,T);\cH^{\sigma-2}\times\cH^{\s-1})$ solution of the inhomogeneous equation 
\begin{equation}\label{linearetotale}
\partial_tV=\big(\mathfrak{A}(\widetilde{V})+
\mathfrak{B}(\widetilde{V})\big)V+\mathfrak{R}(t)
\end{equation}
with $V(0)=V_0:=(Z_0,W_0)\in \mathcal{H}^{\s}\times\cH^{\s}$ satisfying 
\begin{equation}\label{energia-beam}
\|V\|_{L^{\infty}H^{\sigma}}\leq C_re^{C_RT}(\|V_0\|_{H^\sigma}+T\|\mathfrak{R}\|_{L^{\infty}H^{\sigma}}),
\end{equation}
where the constants $C_r$ and $C_{R}$ depend also on $\sigma$.
\end{proposition}
\begin{proof}
We consider the following smoothed version of the initial condition $V_0$
\begin{equation*}
V_0^{\epsilon}:=\chi(\epsilon^{\frac18}|D|)V_0
:=\mathcal{F}^{-1}(\chi(\epsilon^{\frac18}|\xi|)\hat{V}_0(\xi))\,,
\end{equation*}
where $\chi$ is a $C^{\infty}$ function with compact support being equal to one on $[-1,1]$ 
and zero on $\R\setminus [-2,2]$. 
We consider moreover the smoothed, 
homogeneous version of \eqref{linearetotale}, i.e. equation \eqref{beam-smoo}. 
All the solutions $V^{\epsilon}$ are defined 
on a common time interval $[0,T)$ with $T>0$ independent on $\epsilon$, because of Proposition \ref{stima-energia-totale}, see Remark \ref{rem:unifeps}.
We prove that the sequence $V^{\epsilon}$ 
converges to a solution of \eqref{beam-smoo} with 
$\epsilon$ equal to zero both in the initial condition and in the equation.
Let  $0<\epsilon'<\epsilon$, set $V^{\epsilon,\epsilon'}=V^{\epsilon}-V^{\epsilon'}$, then 
\begin{equation}\label{aux}
\partial_tV^{\epsilon,\epsilon'}=[\mathfrak{A}(\widetilde{V})+\mathfrak{B}(\widetilde{V })]V^{\epsilon,\epsilon'}-\epsilon{\bf \Delta}V^{\epsilon,\epsilon'}
+{\bf \Delta}V^{\epsilon}(\epsilon-\epsilon')\,.
\end{equation}
Thanks to the discussion above there exists the flow $\Phi(t)$ of the equation 
\[
\partial_tV^{\epsilon,\epsilon'}=[\mathfrak{A}(\widetilde{V})+\mathfrak{B}(\widetilde{V })]V^{\epsilon,\epsilon'}-\epsilon{\bf\Delta}V^{\epsilon,\epsilon'}\,,
\] 
and it has estimates independent of $\epsilon,\epsilon'$. 
By means of Duhamel formula, we can write the solution of \eqref{aux} in the implicit form
\begin{equation*}
V^{\epsilon,\epsilon'}(t,x)=\Phi(t)(V_0^{\epsilon'}-V_0^{\epsilon})+(\epsilon'-\epsilon)\Phi(t)\int_0^t\Phi(s)^{-1}{\bf\Delta}V^{\epsilon}(s,x)ds\,,
\end{equation*}
using  \eqref{buona-energy-beam} we obtain 
$\|V^{\epsilon,\epsilon'}(t,x)\|_{L^{\infty}H^{\sigma}}\leq 
C (\epsilon-\epsilon')\|V_0\|_{ H^{\sigma}}+(\epsilon-\epsilon')\|V_0^{\epsilon}(t)\|_{H^{\sigma+4}}$ . 
We conclude using {the smoothing estimate} $\|V_0^{\epsilon}\|_{H^{\sigma+4}}\leq \epsilon^{-\frac12}\|V_0\|_{H^{\sigma}}$.  
We eventually proved that we have a well defined 
flow of the equation \eqref{linearetotale} with $\mathfrak{R}(t)=0$. 
The non homogeneous case, i.e. $\mathfrak{R}(t)\neq 0$, 
follows again by Duhamel formula.
\end{proof}

\section{Nonlinear well posedness}
In order to 
prove the existence of the flow of the problem  \eqref{ponteparato} with  $G(t)\equiv0$ and $V(0)=V_0$
we consider, in the spirit of \cite{FIJMPA, FI2018, IN2023}, the iterative scheme  
\begin{equation}\label{iterazione-totale}
\begin{aligned}
&(\mathcal{P})_{1}:=\begin{cases}
\partial_t Z_1=-\ii E\opbw(\xi^2)Z_1, 
\\
\partial_t W_1=-\ii E \opbw(|\xi|)W_1, 
\end{cases} \\
 &(\mathcal{P})_n:= \pa_{t}V_{n}=\mathfrak{A}(V_{n-1})V_{n}+\mathfrak{B}(V_{n-1})V_{n}
 +\mathtt{R}V_{n-1}+\mathcal{R}(V_{n-1})\,,
\quad \mbox{for } n\geq 2\,,
\end{aligned}\end{equation}
with the initial conditions 
$V_n(0,x)=V_0:=(Z_0,W_0)$  any $n\geq 1$.

We prove the following iterative lemma.
\begin{lemma}\label{iterativo}
Let $s\geq s_2>4$  (see \eqref{parametrifissati}),  $V_0=(Z_0,W_0)\in\cH^s\times\cH^s$ and 
\[
r:=2\|V_0\|_{H^{s_1}}.
\] 
There exists $T>0$,  
with 
$T:=T(\|V_0\|_{H^{s_2}})$
small enough, such that for any $n\geq 1$, 
the following statements hold true.

\vspace{0.3em}
\noindent
$(S)_{1,n}-$ The problem $(\mathcal{P})_n$ admits a unique  of solution
$V_{n}=(Z_n,W_n)$ with
\[
\begin{aligned}
V_n\in C^0([0,T); \cH^{s}\times\cH^s)\cap C^1([0,T);\cH^{s-2}\times\cH^{s-1})\,.\quad 
\end{aligned}
\]

\vspace{0.3em}
\noindent
$(S)_{2,n}-$ There exists a constant $C_r\geq 1$ such that, 
if 
\[
R:= 4 C_r\|V_0\|_{H^{s_2}}, \qquad 
M := 4 C_r\|V_0\|_{H^{s}}
\]
for any $1\leq m\leq n$ we have 
\begin{align}
&\|V_m\|_{L^{\infty}H^{s_1}}\leq r\,, \label{bassa}
\\
&\|V_m\|_{L^{\infty}H^{s_2}}\leq R\,,
\quad\|\partial_tV_m\|_{L^{\infty}H^{s_1}}
\leq C_r R\,,\label{media}
\\
&\|V_{m}\|_{L^{\infty}H^{s}}\leq M\,,
\quad  \|\partial_tZ_m\|_{L^{\infty}H^{s-2}}
+\|\partial_tW_{m}\|_{L^{\infty}H^{s-1}}\leq C_rM\,.\label{alta}
\end{align}

\vspace{0.3em}
\noindent
$(S)_{3,n}-$ For any
$1\leq m\leq n$ we have 
\begin{equation}\label{cauchy}
\|V_m-{V_{m-1}}\|_{L^{\infty}H^{s_1}}\leq 2^{-m} r\,.
\end{equation}
\end{lemma}

\begin{proof}
We proceed by induction over $n\in\N$. 
$(S)_{1,1}$ holds because $(\mathcal{P})_1$ 
is just a system of uncoupled linear Schr\"odinger and half-wave equations. 
The linear flows are isometries in $\cH^{s}$, for any $s\in \R$, 
then $(S)_{2,1}$ and $(S)_{3,1}$ follow (recall the definitions of $r, R$ and $M$).

\noindent
Assuming $({S})_{1,n-1}$ and $({S})_{2,n-1}$ we note that \eqref{bassa}, \eqref{media} imply that $V_{n-1}$ satisfies\eqref{palla-beam}, \eqref{palla-wave} with $R\rightsquigarrow C_r R$.
Hence  Proposition \ref{LWP} applies with
$\widetilde{V} \rightsquigarrow V_{n-1}$ and we obtain
the local existence of the solutions of $(\mathcal{P})_{n}$ implying $(S)_{1,n}$.

\noindent
Let us prove $({S})_{2,n}$ starting from \eqref{media}. 
Using \eqref{energia-beam} with $\sigma=s_2$, $\mathfrak{R}(t)=\mathtt{R}V_{n-1}+\mathcal{R}(V_{n-1})$, 
the bounds \eqref{stimaRRR}-\eqref{restilineari} 
and $({S})_{2,n-1}$ we get
{\begin{equation*}
\begin{aligned}
\|V_m\|_{L^{\infty}H^{s_2}}
&\leq 
C_re^{C_RT}(\|V_0\|_{H^{s_2}}+C T (1+\| V_{n-1}\|_{L^{\infty} H^{s_1+1}}) \| V_{n-1} \|_{L^{\infty} H^{s_2}})\\
&\le C_re^{C_RT}(\|V_0\|_{H^{s_2}}+C T (1+R) R)
\end{aligned}
\end{equation*}}
where the constant $C>0$ is the one coming from 
\eqref{stimaRRR}-\eqref{restilineari} 
and depends only  on $s$. 
We choose $T$ small enough in such a way that
{\begin{equation*}
T \exp(C_R T)\leq \frac{1}{2  C\,C_r\,(1+R)}\,, 
\quad  e^{C_R\,T}\leq 2 \,,
\end{equation*}}
therefore we can infer 
\begin{equation*}
\|V_n\|_{L^{\infty}H^{s_2}}\leq \frac{R}{2}+\frac{R}{2}\leq R,
\end{equation*}
because of the choice of $R$ in the statement, hence the $(S)_{2, n}$-\eqref{media} is proved. To prove the bound on the time derivative in $L^{\infty} H^{s_1}$ we observe that
$V_n$ solves the problem $(\mathcal{P})_n$ and we can bound the right hand side of the equation by using the inductive hypothesis and the estimates provided in Proposition \ref{paraparaProp}.  To prove \eqref{alta} we reason similarly: by means of  \eqref{energia-beam} with $\sigma=s$, \eqref{stimaRRR}-\eqref{restilineari} and $({S}_{2,n-1})$ we obtain
\begin{equation*}
\begin{aligned}
\|V_n\|_{L^{\infty}H^s}&\leq C_re^{C_RT}(\|V_0\|_{H^{s}}+CT (1+R) M)\\
\end{aligned}\end{equation*}
We conclude {with the same choice of $T$} as before and with $M$ as in the statement. Concerning the estimate on the time derivative we need to use again the estimates in Proposition \ref{paraparaProp} and the inductive hypothesis. The \eqref{bassa} is a consequence of $({S})_{3,n}$ which we prove below.\\
We set $\check{V}_n:=V_n-V_{n-1}$,  then
\begin{align}
&\partial_t\check{V}_n=[\mathfrak{A}(V_{n-1})+\mathfrak{B}(V_{n-1})](\check{V}_{n})+\mathtt{R}\check{V}_{n-1}+f_n\label{chitam}\\
&f_n:=\mathcal{R}({V}_{n-1})-\mathcal{R}(V_{n-2})+[\mathfrak{A}(V_{n-1})-\mathfrak{A}(V_{n-2})]V_{n-1}+[\mathfrak{B}(V_{n-1})-\mathfrak{B}(V_{n-2})]V_{n-1}
\end{align}
 Since $\mathcal{R}$ is bilinear satisfying the estimate \eqref{stimaRRR}, we have, by using \eqref{bassa},
\begin{equation*}
\begin{aligned}
\|\mathcal{R}(V_{n-1})-\mathcal{R}(V_{n-2})\|_{L^{\infty}H^{s_1}}
\leq C_r\|\check{V}_{n-1}\|_{L^{\infty}H^{s_1}}
\leq C_r2^{-n+1}r\,.
\end{aligned}
\end{equation*}
Furthermore, by using $(S)_{2,n-1}$-\eqref{media}, the contraction estimates \eqref{figaro55}-\eqref{figaro66}, the action Theorem \ref{azione}
and triangular inequality we have 
\begin{equation*}
\begin{aligned}
\| [\mathfrak{A}(V_{n-1})-&\mathfrak{A}(V_{n-2})] V_{n-1}\|_{L^{\infty}H^{s_1}}+\| [\mathfrak{B}(V_{n-1})-\mathfrak{B}(V_{n-2})] V_{n-1}\|_{L^{\infty}H^{s_1}}\leq
\\&
C_R \|\widetilde{V}_{n-1}\|_{L^{\infty}H^{s_1}}
\leq  C_R 2^{-n+1}r\,.
\end{aligned}
\end{equation*}
Hence we proved that 
\[
\| f_n \|_{L^{\infty} H^{s_1}}\leq  C_R 2^{-n+1}r\,.
\]
Applying Proposition \ref{LWP} with $\sigma=s_1$ to  \eqref{chitam}
with initial condition $\check{V}_n(0)=0$ we obtain 
\begin{equation*}
\|\check{V}_n\|_{L^{\infty}H^{s_1}}\leq T C_{R} e^{C_RT} r 2^{-n+1}.
\end{equation*}
We conclude by choosing $T$ such that $e^{C_RT}TC_{R}\leq 1/4.$ \end{proof}

We  use Lemma \ref{iterativo} to prove the following.
\begin{proposition}\label{teototalecomplesso}
Let $s\geq s_2>4$. For any 
$V_0:=(Z_0,W_0)\in\cH^{s}\times\cH^s$ 
there exist 
$
T:=T(\|V_0\|_{H^{s_2}})>0\,,
$
and a  unique solution 
\[
V(t):=(Z(t),W(t))\in C([0,T);\cH^s\times\cH^s)\cap C^1([0,T);\cH^{s-2}\times\cH^{s-1})
\]
of the  system \eqref{ponteparato}
with initial condition  
$V_0:=(Z_0,W_0)\in\cH^{s}\times\cH^s$ at $t=0$.
Moreover the solution depends continuously 
on the initial datum in $\cH^s\times\cH^s$.
\end{proposition}

\begin{proof}
We start by assuming that $G(t)\equiv0$ in \eqref{ponteparato}.
We consider the sequence of solutions 
$V_n$ given by the problems $(\mathcal{P})_n$ in \eqref{iterazione-totale}, 
which exists in virtue of Lemma \ref{iterativo}. 
Again, by Lemma \ref{iterativo}, 
we know that $(V_n)_n$ is a Cauchy sequence in 
$C^0([0,T);\cH^{s_1}\times\cH^{s_1})$ by \eqref{cauchy}
and it is bounded in  $C^0([0,T);\cH^{s}\times\cH^{s})$  by \eqref{alta}. 

By interpolation we deduce that  
it is a Cauchy sequence in $C^0([0,T);\cH^{s'}\times\cH^{s'})$ 
for any $s_1\leq s'<s$. 
The limit of such sequence $V(t)$ is in 
$C^0([0,T);\cH^{s'}\times\cH^{s'})$ and verifies 
$\|V(t)\|_{L^{\infty}H^s}\leq M$ with $M>0$ given by Lemma \ref{iterativo}.  
Analogously one proves that $V(t)$ belongs to 
$C^1([0,T);\cH^{s'-2}\times\cH^{s'-1})$. 
We now prove that $V(t)$  
solves the system \eqref{ponteparato}
in the sense of distributions. 
In order to do that, it is enough to show that
\begin{equation*}
\|\mathcal{A}(V)-\mathcal{A}(V_n)\|_{L^{\infty}H^{s'-2}}\rightarrow 0\,,
\end{equation*}
where $\mathcal{A}(V)$  
is  the r.h.s of \eqref{ponteparato}. 
This is a consequence of the estimates 
\eqref{stimaRRR},\eqref{figaro55}, \eqref{figaro66} the bilinearity of $\mathcal{R}$, 
the linearity of $\mathtt{R}$ and the  bounds  in $(S)_{2,n}$ of Lemma \ref{iterativo} 
with $V(t)\in C^0([0,T); \cH^{s'}\times\cH^{s'})$. 
Up to now we proved that there exists a solution $V(t)$ 
which is in $L^{\infty}([0,T);\cH^s\times\cH^s)\cap C^0([0,T);\cH^{s'}\times\cH^{s'})$.
We now show that it is actually continuous in time with values in $\cH^s\times\cH^s$. 
The idea, following Bona-Smith \cite{BSkdv}, is to prove that 
$V(t)$ is the strong limit of functions in $C^0([0,T); \cH^s\times\cH^s)$. 
We consider the smoothed initial conditions 
$V_{0,N}:=\Pi_NV_0$ where 
$\Pi_N$ is the projector on the first $N$ Fourier modes. 
The initial condition $V_{0,N}$ belongs to $C^{\infty}\times C^{\infty}$. 
Owing to the reasoning above we find a sequence 
$V_N(t)$ of solutions of \eqref{ponteparato}
which are in $C^0([0,T);\cH^{s+1}\times\cH^{s+1})$ 
and which verifies the estimates
\begin{align}
\|V_N(t)\|_{H^s}&\leq C \|V_{0,N}\|_{H^s}\leq C \|V_{0}\|_{H^s}\,,\label{Vbass} 
\\
 \|V_N(t)\|_{H^{s+1}}&\leq C \|V _{0,N}\|_{H^{s+1}}\leq C N\|V_{0}\|_{H^{s}}\,.\label{Valt}
\end{align}
We now show that 
the function
$\check{V}(t):=V(t)-V_N(t)$ goes to zero as $N\to\infty$ in the norm $L^{\infty}H^{s}$. Setting  $\mathfrak{C}(\check{V})=\mathfrak{A}(\check{V})+\mathfrak{B}(\check{V})$, 
we obtain
\begin{equation*}
\partial_t{\check{V}}=\mathfrak{C}(V)\check{V}+\underbrace{\mathtt{R}\check{V}
+\big(\mathfrak{C}(V)-\mathfrak{C}(V_N)\big)V_N+\mathcal{R}(V)-\mathcal{R}(V_N)}_{:=Q}
\end{equation*}
By Proposition \ref{paraparaProp} and the fact that $A_b$ in \eqref{matriceregobeam} is independent of $V$ we have that
\[
\| Q \|_{L^{\infty}H^s}\lesssim_s \|\check{V}\|_{H^{s_1}}\|V_N\|_{H^{s+1}}+\|\check{V}\|_{H^{s}}.
\]

\noindent
We use Proposition \ref{LWP} with $\mathfrak{R}=Q$, the bounds \eqref{Vbass}, \eqref{Valt} and we obtain
\begin{equation*}
\begin{aligned}
\|\check{V}\|_{L^{\infty}H^{s}}&\leq 
Ce^{CT}\|V_0-V_{0,N}\|_{H^s}+TCe^{TC}\|Q\|_{L^{\infty}H^s}
\\
&\leq Ce^{CT}\|V_0-V_{0,N}\|_{H^s}
+TCe^{TC}C(\|\check{V}\|_{H^{s_1}}\|V_N\|_{H^{s+1}}+\|\check{V}\|_{H^{s}})
\\
&\leq Ce^{CT}\|V_0-V_{0,N}\|_{H^s}
+TCe^{TC}C(N^{s_1-s}\|\check{V}\|_{H^{s}}N\|V_N\|_{H^{s}}+\|\check{V}\|_{H^{s}})\,,
\end{aligned}
\end{equation*}
where $C>0$ is a constant depending on $\|V_0\|_{H^s}$. 
Since $s>s_1+1$ 
we can send $N$ to infinity and 
obtain that $\|\check{V}\|_{L^{\infty}H^{s}}$ goes to zero. Therefore $V(t)$ is the strong limit of the sequence $V_N$ in $C^0([0,T); \cH^s\times\cH^s)$. 

\noindent
The uniqueness of the solution $V(t)$ may be obtained by means of similar computations. 
We now prove the continuity of the flow map, here we follow again Bona-Smith \cite{BSkdv}. 
Given a sequence of initial data $(V_0^j)_j$ strongly converging to $V_0$, 
we want to prove that the respective solutions $V^{j}(t)$ 
strongly converge to  $V(t)$ in $\cH^s\times\cH^s$. 
Using the projectors introduced above we get
\begin{equation*}
\begin{aligned}
\|V(t)-V^j(t)\|_{L^{\infty}H^{s}}&\leq \|V(t)-V_N(t)\|_{L^{\infty}H^s}
\\
&+\|V_N(t)-V_N^j(t)\|_{L^{\infty}H^s}+\|V_N^j(t)-V^j(t)\|_{L^{\infty}H^s}\,,
\end{aligned}
\end{equation*}
where $V_N^j(t)$ is the solution with initial condition $V^j_N(0)=\Pi_N V_0^j$.
The three addenda have to be analized separately following 
Bona-Smith \cite{BSkdv}. We can prove that the first and the third summands go to zero as $N$ goes to $+\infty$ by reasoning as above.
One can prove that the second addendum is small taking $j$ sufficiently large with $N$.
For the detailed proof of this last step we refer to 
 \cite{FIJMPA}, which has also similar notations.
 
 \noindent
 The case $G(t)\neq0$ can be treated by means of Duhamel formulation if the problem
 \eqref{sistemaponte} using the flow of the homogeneous equation associated to
 \eqref{sistemaponte}.
\end{proof}
\begin{proof}[{\bf Proof of Theorem \ref{thm:main1}}]
By Proposition \ref{paraparaProp}
we have that system \eqref{sistemaponte} is equivalent to 
system \eqref{ponteparato}. 
By Proposition \ref{teototalecomplesso} 
we have the well posedness of this last system. 
We recover the well posedness for \eqref{sistemaponte} 
recalling Remark \ref{rmk:datiiniziali}, which describes how to pass from the set of complex coordinates to the original ones.
\end{proof}

\vspace{0.5cm}
\gr{Acknowledgements.} 
F. Giuliani wishes to thank F. Gazzola, A. Falocchi, M. Garrione and G. Arioli for many interesting discussions about the analysis of suspension bridges.

F. Iandoli has been partially supported 
by research project PRIN 2017JPCAPN: ``Qualitative and quantitative aspects of nonlinear PDEs" of the 
Italian Ministry of Education and Research (MIUR).
R. Feola and J.E. Massetti  have been  supported by the  research project 
PRIN 2020XBFL ``Hamiltonian and dispersive PDEs" of the 
Italian Ministry of Education and Research (MIUR).  F.  Giuliani, R. Feola and J.E. Massetti have received funding from INdAM-GNAMPA, Project CUP E55F22 000270001.

\vspace{0.5cm}

\gr{Declarations}. Data sharing is not applicable to this article as no datasets were generated or analyzed during the current study.

\noindent
Conflicts of interest: The authors have no conflicts of interest to declare.

\bibliographystyle{plain}

\end{document}